\documentclass%
{amsart}

\usepackage{amssymb,amscd,amsmath}
\usepackage{mathtools}
\usepackage[all,line,arc,curve,color,frame,pdf]{xy}
\usepackage{tikz}
\usepackage{tikz-cd}
\usetikzlibrary{positioning, trees, snakes}
\usepackage[textsize=tiny]{todonotes}
\usepackage{hyperref}
\usepackage[shortlabels]{enumitem}
\usepackage{float}
\usepackage{youngtab}
\usepackage[paper=a4paper, margin=3.5cm]{geometry}
\usepackage{comment}
\usepackage{cleveref}
\usepackage[symbol]{footmisc}

\allowdisplaybreaks

\numberwithin{equation}{section} 
\theoremstyle{plain}
\newtheorem{thm}{Theorem}[section]
\newtheorem{cor}[thm]{Corollary}
\newtheorem{lemma}[thm]{Lemma}
\newtheorem{prop}[thm]{Proposition}

\newtheorem{question}[thm]{Question}
\newtheorem{claim}[thm]{Claim}
\newtheorem{conj}[thm]{Conjecture}

\theoremstyle{remark}
\newtheorem{rem}[thm]{Remark}
\newtheorem{ex}[thm]{Example}
\newtheorem{notation}[thm]{Notation}

\theoremstyle{definition}
\newtheorem{defi}[thm]{Definition}

\newcommand\ML{M\!L}
\newcommand{\SL}{\operatorname{SL}}
\newcommand{\SO}{\operatorname{SO}}
\def\CQ{CQ} %
\def\CS{CS} %
\def\ComplCol{CC} %
\def\fL{\mathcal{L}}
\def\fQ{\mathcal{Q}}

\def\fU{\mathcal{U}}
\def\ot{\otimes}
\newcommand{\Seg}{\operatorname{Seg}}

\DeclareMathOperator\Pf{Pf}
\DeclareMathOperator{\Id}{Id}
\DeclareMathOperator{\Eu}{Eu}

\usepackage[textsize=tiny]{todonotes}

\usepackage{enumitem}

\newcommand\ignore[1]{}

\DeclareMathOperator{\td}{td}
\DeclareMathOperator{\ch}{ch}

\newcommand\CC{{\mathbb{C}}}
\newcommand\PP{{\mathbb{P}}}
\newcommand\RR{{\mathbb{R}}}

\newcommand\ZZ{{\mathbb{Z}}}

\def\N{{\mathbb N}}

\def\cL{{\mathcal L}}

\def\cO{{\mathcal{O}}}

\def\operatorname#1{\mathop{\rm #1}\nolimits}

\def\Pic{\operatorname{Pic}}

\def\length{\operatorname{length}}

\def\rk{\operatorname{rk}}

\def\deg{\operatorname{deg}}

\def\det{\operatorname{det}}

\def\SL{\operatorname{SL}}

\def\SO{\operatorname{SO}}

\newcommand{\Chi}{\ensuremath \raisebox{2pt}{$\chi$}}
\newcommand{\pb}{\ar@{}[dr]|{\text{\pigpenfont J}}}

\newcommand{\xdasharrow}[2][->]{
\tikz[baseline=-\the\dimexpr\fontdimen22\textfont2\relax]{
\node[anchor=south,font=\scriptsize, inner ysep=1.5pt,outer xsep=2.2pt](x){#2};
\draw[shorten <=3.4pt,shorten >=3.4pt,dashed,#1](x.south west)--(x.south east);
}}

\begin{document}
\title[Schubert calculus for Gaussian models and semidefinite programming]{Complete Quadrics: Schubert calculus for Gaussian models and semidefinite programming}

\author[Manivel]{Laurent Manivel}
\address{Institut de Math\'ematiques de Toulouse ; UMR 5219, Universit\'e de Toulouse \& CNRS, F-31062 Toulouse Cedex 9, France}
\email{manivel@math.cnrs.fr}
\author[Micha{\l}ek]{Mateusz Micha{\l}ek}
\address{
University of Konstanz, Germany, Fachbereich Mathematik und Statistik, Fach D 197
D-78457 Konstanz, Germany
}
\email{mateusz.michalek@uni-konstanz.de}
\author[Monin]{Leonid Monin}
\address{University of Bristol, School of Mathematics, BS8 1TW, Bristol, UK}
\email{leonid.monin@bristol.ac.uk}
\author[Seynnaeve]{Tim Seynnaeve}
\address{University of Bern, Mathematical Institute, Sidlerstrasse 5, 3012 Bern, Switzerland}
\email{tim.seynnaeve@math.unibe.ch}
\author[Vodi\v{c}ka]{Martin Vodi\v{c}ka}
\address{Max Planck Institute for Mathematics in the Sciences, 04103 Leipzig, Germany}
\email{martin.vodicka@mis.mpg.de}

\begin{abstract}
We establish connections between: the maximum likelihood degree (ML-degree) for linear concentration models, the algebraic degree of semidefinite programming (SDP), and Schubert calculus for complete quadrics. We prove a conjecture by Sturmfels and Uhler on the polynomiality of the ML-degree. We also prove a conjecture by Nie, Ranestad and Sturmfels providing an explicit formula for the degree of SDP. The interactions between the three fields shed new light on the asymptotic behaviour of enumerative invariants for the variety of complete quadrics. We also extend these results to spaces of general matrices and of skew-symmetric matrices. 
\end{abstract}

\thanks{ }
\maketitle

\section{Introduction}

\subsection*{Maximum likelihood degree and quadrics}
Although this paper is mainly about enumerative geometry and symmetric functions, the main motivations come from algebraic statistics and multivariate Gaussian models. These are generalizations of the well-known Gaussian distribution to higher dimensions. In the one-dimensional case, in order to determine a Gaussian distribution on $\RR$, one needs to specify its mean $\mu\in \RR$ and its variance $\sigma \in \RR_{>0}$. In the $n$-dimensional case, the mean is a vector $\mu\in \RR^n$, and the second parameter is a positive-definite $n\times n$ covariance matrix $\Sigma$. The corresponding Gaussian distribution on $\RR^n$ is given by
\[f_{\mu,\Sigma}(x):=\frac{1}{\sqrt{(2\pi)^n\det\Sigma}}e^{-\frac{1}{2}(x-\mu)^T\Sigma^{-1}(x-\mu)},\]
where $^T$ denotes the transpose.
Equivalently to determining it by $\mu$ and $\Sigma$, one may represent the distribution by $\mu$ and the \emph{concentration matrix} $K:=\Sigma^{-1}$, which is also positive definite. Our primary interest lies in \emph{linear concentration models}, i.e.~statistical models which assume that $K$ belongs to a fixed $d$-dimensional space $\fL$ of $n\times n$ symmetric matrices. These were introduced by Anderson half a century ago \cite{MR0277057}. In particular, this means that $\Sigma$ should belong to the set $\fL^{-1}$ of inverses of matrices from $\fL$.

In statistics, typically one gathers data as sample vectors $x_1,\dots,x_s\in \RR^n$. This allows to estimate the mean $\mu$ as the mean of the $x_i$'s. Furthermore, each $x_i$ provides a 
matrix $\Sigma_i:=(x_i-\mu)(x_i-\mu)^T$. Next, one considers the \emph{sample covariance matrix} $S$, that is the mean of the $\Sigma_i$'s. Note that in most situations, it is not true that $S\in \fL^{-1}$. The aim is then to find $\Sigma$ that best explains the observations. From the point of view of statistics, it is natural to maximize the likelihood function 
\[f_{\mu,\Sigma}(x_1)\cdots f_{\mu,\Sigma}(x_s),\]
that is, to find a positive definite matrix $\Sigma\in \fL^{-1}$ for which the above value is maximal. Classical theorems in statistics assert that the solution to this optimization problem is essentially geometric \cite[Theorem 3.6, Theorem 5.5]{brown1986fundamentals}, \cite[Theorem 4.4]{MSUZ}. Namely, under mild genericity assumptions, the optimal $\Sigma$ is the unique positive definite matrix in $\fL^{-1}$ that maps to the same point as $S$ under projection from $\fL^{\perp}$. 

This is one of the main reasons why the complex variety that is the Zariski closure of $\fL^{-1}$ (which abusing notation we also denote by $\fL^{-1}$) and the rational map $\pi$ defined as the projection from $\fL^{\perp}$ are intensively studied in algebraic statistics. Note that for generic $\fL$, and after projectivization, $\pi$ becomes a finite map. The following is the  central definition of the article.

\begin{defi}[ML-degree]
	The \emph{ML-degree} of a linear concentration model represented by a space $\fL$ is the degree of the projection from the space $\fL^\perp$ restricted to the variety $\fL^{-1}$.
\end{defi}

The ML-degree is the  basic measure of the complexity of the model. When $\fL$ is a generic space, the ML-degree only depends on the size $n$ of the symmetric matrices and on the (affine) dimension $d$ of $\fL$.  By a theorem of Teissier \cite{T1, T2} (cf.~\cite[Corollary 2.6]{MMW2020}) or Sturmfels and Uhler \cite[Theorem 1]{StUh}, the ML-degree equals the degree 
of the variety $\fL^{-1}$. Following Sturmfels and Uhler \cite{StUh} we denote it by $\phi(n,d)$. We refer algebraists interested in statistics to \cite{drton2009lectures} for more information about the subject.

\begin{defi} \label{def:phi}
	For $n \in \ZZ_{>0}$ and $1 \leq d \leq \binom{n+1}{2}$, we define $\phi(n,d)$ to be the degree of the variety $\fL^{-1}$, where $\fL$ is a general $d$-dimensional linear subspace of $S^2\CC^n$.
\end{defi}

Thus, our main result concerns a very basic algebro-geometric object: the degree of the variety obtained by inverting all symmetric matrices in a general linear space. In Section \ref{sec:pol}, we confirm the following conjecture of Sturmfels and Uhler \cite[p.~611]{StUh}, \cite[Conjecture 2.8]{MMW2020}. This theorem is actually a corollary of a similar polynomiality result for the \emph{algebraic degree of semidefinite programming}, which will be discussed in the next section.

\begin{thm}\label{phiintro}
For any fixed positive integer $d$, the ML-degree $\phi(n, d)$ is  polynomial in $n$. 
\end{thm}

Astonishingly, it appears that the numbers $\phi(n,d)$ were studied for the last 150 years! 
In 1879 Schubert presented his fundamental results on quadrics satisfying various tangency conditions \cite{Schubert1879}. His contributions shaped the field of enumerative geometry, inspiring many mathematicians for centuries to come. A nondegenerate quadric being given, the set of its tangent hyperplanes (its projective dual, in modern language) is nothing else than 
the inverse quadric. This implies that  $\phi(n, d)$ is also the solution to the following enumerative problem: 

\begin{quote}
{\it What is the number of nondegenerate quadrics in $n$ variables, passing through ${{n+1}\choose 2}-d$ general points and tangent to $d-1$ general hyperplanes?}
\end{quote}

\smallskip
In modern language, such problems can be solved by performing computations in the cohomology ring of the \emph{variety of complete quadrics}. This is now a classical topic with many beautiful results \cite{Semple1,Semple2,Tyrrell,Vainsencher,DeConciniProcesi1, DeConciniProcesi2, LaksovCompleteQuadrics,CGMP,Thaddeus2,Massarenti1}.  In particular, the cohomology ring has been described by generators and relations, and algorithms have been devised that allow to compute any given intersection number. But this only applies for $n$ fixed. Algebraic statistics suggested to change the perspective and to fix $d$ instead of $n$. This explains, in a way, why the polynomiality property of $\phi (n,d)$ is only proved now.

\subsection*{Semidefinite programming and projective duality} 

The second domain of mathematics that inspired our research is semidefinite programming (SDP),  a very important and effective subject in optimization theory. The goal is to study linear optimization problems over spectrahedra. This subject is a direct generalization of linear programming, that is optimization of linear functions over polyhedra. For a short introduction to the topic we refer to \cite[Chapter 12]{book}.

The coordinates of the optimal solution for an SDP problem, defined over rational numbers, are algebraic numbers. Their algebraic degree is governed by \emph{the algebraic degree of semidefinite programming}. For more information we refer to the fundamental article \cite{NRS}. To stress the importance of this degree let us 
just quote this paper:

\begin{quote}
\emph{"The algebraic degree of semidefinite programming addresses the computational complexity at a fundamental level. To solve the semidefinite programming exactly essentially reduces to solve a class of univariate polynomial equations whose degrees are the algebraic degree.}"
\end{quote} 

Let us provide a precise definition of the algebraic degree of SDP, in the language of algebraic geometry, without referring to optimization. (However, the fact that this definition is correct is actually a nontrivial result \cite[Theorem 13]{NRS}.)

\begin{defi} \label{def: delta} For $0 < m < \binom{n+1}{2}$ and $0<r<n$, let $\fL \subset S^2\CC^n$ be  a general linear space of symmetric matrices, of (affine) dimension $m+1$, and let $SD_{m}^{r,n}\subset \mathbb{P}(\fL)$ denote the projectivization of the cone
	of matrices of rank at most $r$ in $\fL$. The \emph{algebraic degree of semidefinite programming} $\delta(m,n,r)$ is the degree of the projective dual $(SD_{m}^{r,n})^*$ of S$D_{m}^{r,n}$ if this dual is a hypersurface, and zero otherwise.
\end{defi}

Projective duality is a very classical topic, to which a huge literature has been devoted. Computing the degree of a dual
variety is well-known to be very hard, especially when the variety in question is singular, which is often the 
case for our  $SD_{m}^{r,n}$. 
Nevertheless, Ranestad and Graf von Bothmer \cite{BothmerRanestad} suggested to use conormal varieties, and managed to obtain an algebraic expression for $\delta(m,n,r)$ in terms of what we call the {\it Lascoux coefficients}. 
These are integer coefficients that govern the Segre classes of the symmetric square 
of a given vector bundle; algebraically, they are defined by the formal identity 
$$\prod_{1\le i\le j\le s}\frac{1}{1-(x_i+x_j)}=\sum_I\psi_Is_{\lambda(I)}(x_1,\ldots ,x_s),$$
where the sum is over the increasing sets $I=(i_1<i_2<\cdots <i_s)$ of nonnegative integers,
$\lambda(I)=(i_s-s+1,\ldots , i_2-1, i_1)$ is the associated partition, and $s_{\lambda(I)}(x_1,\ldots ,x_s)$
the corresponding Schur function in the variables $x_1,\ldots , x_s$. These coefficients 
were introduced and studied in \cite{LaksovCompleteQuadrics, pragacz1988enumerative}, whose influence on our work cannot be underestimated. Graf von Bothmer and Ranestad found a formula for $\delta(m,n,r)$ in terms of the Lascoux coefficients (see \Cref{thm:delta}).
Diving into the combinatorics of those coefficients, 
in Section \ref{sec:pol} we prove the following polynomiality result:
\begin{thm}\label{deltaintro}
	For any fixed $m,s >0$, the function $\delta(m,n,n-s)$ is a polynomial in $n$.
\end{thm}
Moreover, in Section \ref{sec:NRS} we confirm \cite[Conjecture 21]{NRS}, providing another explicit formula for $\delta(m,n,r)$, and another proof of the above theorem.
\begin{thm}\label{nrsIntro} (NRS, Conjecture 21)
Let $m,n,s$ be positive integers. Then
\[\delta(m,n,n-s)=\sum_{\sum I\le m-s} (-1)^{m-s-\sum I}\psi_I b_I(n)\binom {m-1} {m-s-\sum I}\]
where the sum goes trough all sets of nonnegative integers of cardinality $s$. 
\end{thm}

In this formula, $\Sigma I=i_1+\cdots +i_s$, and $b_I(n)$ is a polynomial function  of $n$ defined in Section \ref{sec:NRS}. Actually, $b_I(n)$ is obtained by evaluating a Q-Schur polynomial on $n$ identical variables; by the work of Stembridge \cite{Stembridge}, it counts certain
shifted tableaux of shape determined by $I$, numbered by integers not greater than $n$. 
This also proves the polynomiality of the ML-degree, since elementary relations in the cohomology ring of the variety of complete quadrics imply the fundamental identity (see \Cref{FundamentalRelation}):
$$\phi(n,d)=\frac{1}{n}\sum_{\binom{s+1}{2}\le d} s\delta(d,n,n-s).$$

\smallskip 
Our approach also applies to linear spaces of general square matrices or of skew-symmetric matrices,
and allows to obtain closed formulas for the dual degrees of their determinantal loci. In particular,
the analogues of the Nie-Ranestad-Sturmfels conjecture hold true. This means in particular that these
dual degrees can be computed by the class formula, essentially as if those determinantal loci were 
always smooth (which is certainly not the case in general!). Since the dual degree is well-known to
be extremely sensitive to singularities, it would be very interesting to have a conceptual explanation
of this phenomenon. 

\subsection*{Computations}
The proofs of the Theorems \ref{phiintro}, \ref{deltaintro}, \ref{nrsIntro} give rise to explicit polynomial formulas for $\delta(m,n,n-s)$ and $\phi(n,d)$, which can be evaluated using software.
We illustrate this on a Sage worksheet, available on
\begin{center}
	 \href{https://mathrepo.mis.mpg.de/MLdegreeCompleteQuadrics/index.html}{https://mathrepo.mis.mpg.de/}.
\end{center}
So far the exact formula for $\phi(n,d)$ was only known for $d\leq 5$ \cite{MR3406441,MR1161918,StUh,MMW2020}. We compute it explicitly for  $d\leq 50$, 
confirming in particular \cite[Conjecture 5.1]{MMW2020}.

\section*{Acknowledgements}
We would like to thank Andrzej Weber and Jaros{\l}aw Wi{\'s}niewski for many interesting discussions. 
We are grateful to Piotr Pragacz for important remarks about the article.

\section{Notation and preliminaries}

\subsection{Partitions, Schur polynomials}
\begin{notation}
For a set of nonnegative integers $I=\{i_1,\dots,i_r\}$, we assume $i_1<\cdots <i_{r-1}<i_r$ and we define the corresponding partition 
\[\lambda(I):=(i_r-(r-1),i_{r-1}-(r-2),\dots,i_2-1,i_1).\]

Analogously, for a partition $\lambda=(\lambda_1,\dots,\lambda_r)$, which means $\lambda_1\ge\cdots\ge \lambda_r\ge 0$ (zeroes are allowed), we define the corresponding set \[I(\lambda):=\{\lambda_r,\lambda_{r-1}+1,\dots,\lambda_2+r-2,\lambda_1+r-1\}.\]
The length of a partition $\lambda$, (i.e.\ the number of nonzero entries) will be denoted by $\length(\lambda)$. By $|\lambda|$ we denote the size of the partition $\sum_{i=1}^r \lambda_r$. 
By $\tilde\lambda$ we denote the partition conjugate to $\lambda$, e.g.~$\widetilde{(3,1)}=(2,1,1)$.

We will abbreviate $\{0,\dots,n-1\}$ to $[n]$.
Let $\sum I:= i_1+\dots +i_r$ denote the sum of elements of $I$ and $\# I=r$ its cardinality. 
For two sets $I=\{i_1,\dots, i_r\}$ and $J=\{j_1,\dots,j_r\}$ we say that $I\leq J$ if $i_k\leq j_k$ for all $1\leq k\leq r$. 
\end{notation}

\begin{notation}\label{def:Schur}
For a partition $\lambda$ we denote by $s_\lambda$ the corresponding Schur polynomial.
\end{notation}

\begin{defi}\label{def:sIJ}
Let $I,J$ be two sets of nonnegative integers of cardinality $r$. We define the 
numbers $s_{I,J}$ to be the unique integers which satisfy the polynomial equation
\[s_{\lambda(I)}(x_1+1,\dots,x_r+1)=\sum_{J\le I} s_{I,J}s_{\lambda(J)}(x_1,\dots,x_r).\]
\end{defi}

Note that since $s_{\lambda(I)}$ is a homogeneous polynomial, we also have the identity
\[s_{\lambda(I)}(x_1-1,\dots,x_r-1)=\sum_{J\le I} (-1)^{\sum I-\sum J}s_{I,J}s_{\lambda(J)}(x_1,\dots,x_r).\]
As a consequence, the triangular matrices $(s_{I,J})_{I,J}$ and $((-1)^{\sum I-\sum J}s_{I,J})_{I,J}$ 
are inverses of each other. 

\begin{lemma}\label{shiftedcomplete}
Let $I=\{i_1,\dots,i_r\}$ and $J=\{j_1,\dots,j_r\}$ be two sets of nonnegative integers. Let $M_{I,J}=(m_{kl})$ be the $r\times r$ matrix with $m_{kl}=\binom {i_k}{j_l}$. Then
\begin{itemize}
\item[a) ]$s_ {I,J}=\det (M_{I,J})$
\item[b) ]$s_{(d)}(x_1+1,\dots,x_r+1)=\sum_{i=0}^d \binom{d+r-1}{d-i}s_{(i)}(x_1,\dots,x_r)$
\end{itemize}
\end{lemma}
\begin{proof}
Part a) is proved in \cite[Section I.3, example 10]{Macdonald}.
In particular, it implies \[s_{[r+d],[r+i]}=\binom{d+r-1}{r+i-1}=\binom{d+r-1}{d-i}.\]
From this, the equation in part b) becomes the defining equation for $s_{I,J}$.
\end{proof}

\subsection{Lascoux coefficients}
\begin{defi}\label{def:psi} 
We define the \emph{Lascoux coefficients} $\psi_I$ by the following formula:
\[
		s_{(d)}(\{x_i + x_j \mid 1 \leq i \leq j \leq r\})= \sum_{\tiny \begin{matrix} \lambda(I) \vdash d \\ \# I= r \end{matrix}}\psi_I s_{\lambda(I)}(x_1,\ldots,x_r),
		\]
Here $s_{(d)}$ is the complete symmetric polynomial of degree $d$, in the $\binom{r+1}{2}$ variables $x_i+x_j$. Hence, the coefficients $\psi_I$ appear in the expansion  in the Schur basis, of the complete symmetric polynomial evaluated at sums of variables.

Equivalently, the Lascoux coefficients appear in the expansion of the $d$-th Segre class of the second symmetric power of any rank $r$ vector bundle in terms of its Schur classes. In particular, for  the universal bundle $\fU$ over the Grassmannian $G(r,n)$ for $n\geq r+d$,
\[
		Seg_d(S^2\fU) = \sum_{\tiny \begin{matrix} \lambda(I) \vdash d \\ \# I= r \end{matrix}}{\psi_I \sigma_{\lambda(I)}},
		\]
where $\sigma_{\lambda}$ denote the Schubert classes in the Chow ring of the Grassmannian.
(For $r \leq n < r+d$ the identity is still true, but some of the Schubert classes 
$\sigma_{\lambda(I)}$ will be zero.)
\end{defi}

\begin{ex}
Let us consider $r=2$ and $n=4$, i.e.~the Grassmannian $G(2,4)$. The rank two universal vector bundle $\fU$ has two Chern roots $x_1,x_2$. Recall that the cohomology ring of $G(2,4)$ is six-dimensional with basis corresponding to Young diagrams contained in the $2\times 2$ square. We have formal equalities:
$$x_1+x_2=-\yng(1),\quad x_1\cdot x_2=\yng(1,1).$$
The Chern roots of $S^2\fU$ are $2x_1,x_1+x_2, 2x_2$. Computing the elementary symmetric polynomials in those we obtain the three respective Chern classes:
$$-3\,\yng(1),\quad 2\,\yng(2)+6\,\yng(1,1), \quad-4\,\yng(2,1).$$
By inverting the Chern polynomial we obtain the Segre classes:
$$3\,\yng(1),\quad \mathbf{7}\,\yng(2)+\emph{3}\,\yng(1,1),\quad 10\,\yng(2,1),\quad 10\,\yng(2,2).$$
Their coefficients are the Lascoux coefficients, namely:
$$\psi_{0,2}=3, \quad \psi_{0,3}=\mathbf{7}, \quad  \psi_{1,2}=\emph{3}, 
\quad \psi_{1,3}=10,\quad  \psi_{2,3}=10.$$
We use boldface and emphasis above and below to indicate the same numbers.
We may also compute them by expanding complete symmetric polynomials, where now $x_1,x_2$ are simply formal variables.
$$s_{(2)}(2x_1,x_1+x_2,2x_2)=7x_1^2+7x_2^2+10x_1x_2=$$
$$=7(x_1^2+x_1x_2+x_2^2)+3x_1x_2=\mathbf{7}s_{(2,0)}(x_1,x_2)+\emph{3}s_{(1,1)}(x_1,x_2).$$
\end{ex}

We note that Lascoux coefficients  appear in many publications with different notations. In particular
one needs to be careful with the shift: $\psi_{\{j_1,\ldots,j_r\}}$ as defined above equals $\psi_{\{j_1+1,\ldots,j_r+1\}}$ in \cite{BothmerRanestad}. On the other hand our notation is consistent with \cite{LLT,NRS}.

The lemma below gives a closed formula for the Lascoux coefficients, in terms of Pfaffians.
\begin{lemma}[{\cite[Prop.~7.12]{pragacz1988enumerative}}]\label{pfaffianpsi}
	Let $I=\{i_1,\dots,i_r\}$ be a set of nonnegative integers. If $r=1,2$ then $\psi_I$ is given by 
	$\psi_{\{i\}} = 2^i$ and $\psi_{\{i,j\}}=\sum_{k=i+1}^j\binom{i+j}{k}$ respectively.
	 For $r>2$, $\psi_I$ can be computed as \[\psi_{I}=\Pf(\psi_{ \{i_k,i_l\}})_{0<k<l\le n} \text{ for even }r,\]
	\[\psi_{I}=\Pf(\psi_{\{i_k,i_l\}})_{0\le k<l\le n} \text{ for odd }r,\]
	where $\psi_{\{i_0,i_k\}}:=\psi_{ \{i_k\}}.$
\end{lemma}

\subsection{SDP-degree and ML-degree} Recall the Definitions \ref{def:phi} and \ref{def: delta} of the ML-degree $\phi(n,d)$ and the SDP-degree $\delta(m,n,r)$.
\begin{rem} \label{rem:defExtension}
	For our polynomiality results in \Cref{sec:pol}, it will be useful to extend the definitions of $\phi$ and $\delta$ as follows:
	\begin{itemize}
		\item For $d>\binom{n+1}{2}$, we put $\phi(n,d)=0$.
		\item For $m=0$ and $r<n$, we define $\delta(0,n,r)=0$.
		\item For $m \geq \binom{n+1}{2}$ or $s\geq n$, we put $\delta(m,n,n-s)=0$, with one exception: in the case $m = \binom{n+1}{2}$ and $s = n$, we define $\delta(\binom{n+1}{2},n,0)=1$. See also \Cref{rem:deltaExtend}.
	\end{itemize}
	Now $\phi(n,d)$ is defined for all $n,d>0$, and $\delta(m,n,n-s)$ is defined for all $m,n,s>0$.
\end{rem}

For later reference, we recall the description of $\delta(m,n,r)$ in terms of the bidegrees of a conormal variety:
\begin{thm}[{\cite[Theorem 10]{NRS}}] \label{thm:multiDegree}
	Let $Z_r \subseteq \PP(S^2V) \times \PP(S^2V^*)$ be the conormal variety to the variety $SD^r \subseteq \PP(S^2V)$ of matrices of rank at most $r$. Explicitely, $Z_r$ consisit of pairs of symmetric matrices $(X,Y)$,
	up to scalars, with $\rk X \leq r$, $\rk Y \leq n-r$, and $X\cdot Y = 0$. Then the multidegree of $Z_r$ is given by
	\[
	[Z_r] = \sum_{m=0}^{\binom{n+1}{2}}\delta(m,n,r)H_1^mH_{n-1}^{\binom{n+1}{2}-m}.
	\]
\end{thm}

Here $H_1$ and $H_{n-1}$ denote the pull-backs of the hyperplane classes from 
$\PP(S^2V)$ and $ \PP(S^2V^*)$.

\begin{rem} \label{rem:dualityDelta}
	From this description, we immediately get the following duality relation (see also \cite[Proposition 9]{NRS}) 
	\[
	\delta(m,n,n-s)=\delta(\binom{n+1}{2}-m,n,s).
	\]
\end{rem}

\section{ML-degrees via complete quadrics}

\subsection{Spaces of complete quadrics}
Let $V$ be a $n$-dimensional vector space over $\CC$. The space of complete quadrics $\CQ(V)$ is a particular compactification of the space of smooth quadrics in $\PP(V)$, or equivalently, of the space of invertible symmetric matrices 
(up to scalar) $\PP(S^2(V))^\circ \subset \PP(S^2(V))$. The space of complete quadrics $\CQ(V)$ has several equivalent descriptions, below we will describe some of them. For more information we refer the reader to \cite{LaksovCompleteQuadrics,Thaddeus2,Massarenti1}.

For $A\in S^2(V)$ let $\bigwedge^k A \in S^2(\bigwedge^k  V)$ be the corresponding form on $\bigwedge^k V$. If we view $A$ as a symmetric matrix, then $\bigwedge^k A$ is the matrix of $k \times k$ minors of $A$. In particular, $\bigwedge^{n-1}  A$  is the inverse of $A$ up to scaling (by the determinant).

\begin{defi}\label{def:CQ1}
The space of complete quadrics $\CQ(V)$ is the closure of $\phi(\PP(S^2(V))^\circ)$, where
$$
\phi : \PP(S^2(V))^\circ  \to \PP(S^2(V)) \times \PP\left( S^2\left(V\wedge V\right) \right) \times\ldots\times \PP \left(S^2\left(\bigwedge^{n-1} V\right) \right)
$$
is given by
$$
[A]\mapsto \left([A],[\bigwedge^2 A],\ldots, [\bigwedge^{n-1} A]\right).
$$
To simplify the notation we will also denote $CQ(V)$ by $\CQ_n$.
\end{defi}

The natural projection to the $j$-th factor 
induces a regular map $$\pi_j:\CQ(V) \to  \PP\left(S^2\left(\bigwedge^{j} V\right) \right).$$ 
In particular, the map $\pi_1:\CQ(V)\to \PP(S^2(V))$, which is an isomorphism on $\phi( \PP(S^2(V))^\circ)$,
can be described as a sequence of blow-downs. This  provides the second description of the space of complete quadrics.

\begin{defi}\label{def:CQ2}
The space of complete quadrics $\CQ(V)$ is the successive blow-up of $\PP(S^2(V))$:
$$
\CQ(V) = Bl_{\widetilde D^{n-1}} Bl_{\widetilde D^{n-2}} \ldots Bl_{D^1} \PP(S^2(V)),
$$
where $\widetilde D^{i} $ is the proper transform of the space of symmetric matrices of rank at most $i$ 
under the previous blow-ups.
\end{defi}

\begin{thm}[{\cite[Theorem 6.3]{Vainsencher}}]
Definitions \ref{def:CQ1} and \ref{def:CQ2} of the space of complete quadrics are equivalent.
\end{thm}
The space of complete quadrics admits other descriptions, we would like to mention two of them. The first one describes the space of complete quadrics as an equivariant compactification of the space of invertible symmetric matrices, which is a spherical homogeneous space:
$$
\PP(S^2(V))^\circ \simeq \SL_n/N(\SO_n),
$$
where $N(\SO_n)$ is the normalizer of $\SO_n$. The second description realises the space of complete quadrics as a subvariety of    the Kontsevich moduli space of stable maps to the Lagrangian Grassmannian, see for details \cite{Thaddeus2,MMW2020}.

The space of complete quadrics has two natural series of special classes of divisors. The first series 
 $S_1,\ldots,S_{n-1}$ is given by the classes of the (strict transforms) of the exceptional divisors
 $E_1,\ldots,E_{n-1}$ of the successive blow-ups in Definition \ref{def:CQ2} (which are precisely 
 the $\SL_n$-invariant prime divisors on $\CQ(V)$). The second series $L_1,\ldots, L_{n-1}$ is
 obtained by pulling back the hyperplane classes by the projections $\pi_1,\ldots , \pi_{n-1}$.

\begin{prop}\label{prop:relations}
The classes $L_1,\ldots, L_{n-1}$ form a basis of $\Pic(\CQ(V))$, in which the classes 
 $S_1, \ldots, S_{n-1}$ are given by the 
  relations
$$
S_i= -L_{i-1}+2L_i-L_{i+1},
$$
with $L_0=L_n:=0$.
\end{prop}
\begin{proof}
	These relations were already known to Schubert \cite{SchubertAllgemein}. For a modern treatment, see for example \cite[Proposition 3.6 and Theorem 3.13]{Massarenti1}.
\end{proof}

The  inverse relations are given by the $(n-1)\times(n-1)$ matrix $M$ given by:
$$
M_{i,j} = \min(i,j) - \frac{ij}{n},
$$
in particular we have:
$$
nL_1 = (n-1)S_1 +(n-2)S_2 + \ldots + S_{n-1},
$$
\begin{equation}
nL_{n-1} = S_1 +2S_2 + \ldots + (n-1)S_{n-1}. \label{eq:Ln-1}
\end{equation}

\subsection{Intersection theory}
We are ready to relate the computation of the $\ML$-degree and of 
the algebraic degree of semidefinite programming to the intersection theory of $\CQ(V)$.

\begin{prop} \label{prop:PhiDeltaChow}
The $\ML$-degree and the algebraic degree of semidefinite programming can be computed as the following 
 intersection numbers on $\CQ(V)$:
$$
\phi(n,d) = \int_{CQ_n}L_1^{\binom{n+1}{2}-d} L_{n-1}^{d-1},
$$
$$
\delta(m,n,r) = \int_{E_r}L_1^{\binom{n+1}{2}-m-1}L_{n-1}^{m-1}=
\int_{CQ_n}S_rL_1^{\binom{n+1}{2}-m-1}L_{n-1}^{m-1}.
$$
\end{prop}
\begin{proof}
Since the morphisms $\pi_1$ and $\pi_{n-1}$ resolve the inversion map $\PP(S^2V) \dashrightarrow \PP(S^2V^*)$, we can compute the degree of $\fL^{-1}$, for $\fL\subseteq \PP(S^2V^*)$ a general $d-1$-dimensional linear subspace, as $\pi_1^*(H_1^{\binom{n+1}{2}-d})\pi_{n-1}^*(H_{n-1}^{d-1})$, where $H_1$ and $H_{n-1}$ are hyperplane classes in $\PP(S^2V)$ and $\PP(S^2V^*)$ respectively. Indeed the factor $H_1^{\binom{n+1}{2}-d}$
imposes $\binom{n+1}{2}-d$ linear conditions to matrices in $\PP(S^2V)$, defining a linear subspace 
$\PP(\fL)$ of projective dimension $d-1$, while the factor $H_{n-1}^{d-1}$
imposes $d-1$ linear conditions on $\PP(\fL^{-1})$, hence computing the degree $\phi(n,d)$.

	For $\delta(m,n,r)$, the main observation is that $E_r$ is birationally isomorphic to the conormal variety $Z_r$. Indeed, the inductive properties of the spaces of complete quadrics imply that $E_r$ can be described as a space of relative complete quadrics; more precisely 
	$$E_r=\CQ_r(\fU)\times_{G(r,n)} \CQ_{n-r}(\fQ^*).$$
The equality above may be derived from the quotient construction of the variety of complete quadrics, cf.~\cite{MMW2020}, \cite{Thaddeus2}.
	In particular $E_r$ is birationally isomorphic to $Y_r=\PP(S^2\fU)\times_{G(r,n)}\PP(S^2\fQ^*)$. As observed for example in \cite{BothmerRanestad}, this is also a smooth model for the conormal variety $Z_r$. As the divisors $L_i$ are base point free,
by \Cref{thm:multiDegree}
$$\delta(m,n,r)=\int_{Z_r}L_1^{\binom{n+1}{2}-m-1}L_{n-1}^{m-1}=
\int_{Y_r}L_1^{\binom{n+1}{2}-m-1}L_{n-1}^{m-1}$$
can be computed on $E_r$, and this implies our claim. 
\end{proof}

Since $S_r$ projects in $\PP(S^2(V))$ to the locus of matrices of rank at most $r$, we must have 
$S_rL_1^{\binom{n+1}{2}-m-1}=0$ when $m$ is smaller that the codimension of this locus. Similarly
$S_rL_{n-1}^{m-1}=0$ when $m$ is big enough. One can deduce that the following 
 \emph{Pataki inequalities} are necessary and sufficient conditions  for $\delta(m,n,r)$ to be nonzero
 \cite[Proposition 5 and Theorem 7]{NRS}: 
\begin{equation}\label{eq:Pataki}
{{n-r+1}\choose{2}}\leq  m\leq {{n+1}\choose{2}}-{{r+1}\choose{2}}.
\end{equation}

We can then use \eqref{eq:Ln-1} to write the ML-degree in terms of the SDP-degree:

\begin{cor}\label{FundamentalRelation}
For any $n,d >0$, the following fundamental relation does hold:
$$\phi(n,d)=\frac{1}{n}\sum_{1 \leq \binom{s+1}{2} \leq d}{s\delta(d,n,n-s)}.$$
\end{cor}
\begin{proof}
First, for $d > \binom{n+1}{2}$, we have both sides equal to $0$, and for $d=\binom{n+1}{2}$ both sides equal $1$ (see \Cref{rem:defExtension}), so the relation holds.
For $1 \leq d <\binom{n+1}{2}$, we can write
\begin{align}\label{eq:phitodeltaAlt}
\phi(n,d) &= \int_{CQ_n}L_1^{\binom{n+1}{2}-d} L_{n-1}^{d-1} \nonumber\\
&=  \frac{1}{n}\int_{CQ_n}L_1^{\binom{n+1}{2}-d-1} L_{n-1}^{d-1} \sum_{s=1}^{n-1}{sS_{n-s}} \\
&= \frac{1}{n}\sum_{s=1}^{n-1}{s\delta(d,n,n-s)}= \frac{1}{n}\sum_{1 \leq \binom{s+1}{2} \leq d}{s\delta(d,n,n-s)}.\nonumber
\end{align}
The last equality follows from the Pataki inequalities, since  $\delta(d,n,n-s)=0$ whenever $\binom{s+1}{2}>d$.\qedhere
\end{proof}

All our computations can now be reduced to the intersection theory of the Grassmannians.

\begin{thm}[{\cite[Theorem 1.1]{BothmerRanestad}}] \label{thm:delta}
	For $0 < m < \binom{n+1}{2}$ and $0<r<n$,
	\[
	\delta(m,n,r) = \sum_{\substack{I \subset [n] \\ \# I=n-r \\ \sum{I}=m-n+r}}{\psi_I \psi_{[n] \setminus I}}
	\]
\end{thm}
\begin{proof}[Idea of proof]
As already mentioned, $\delta(m,n,r)$ can be computed as an intersection number on $Y_r$, which is a 
fiber bundle over the Grassmannian $G(r,n)$. By push-forward, one obtains 
	\[
	\delta(m,n,r) = \int_{G(r,n)}\Seg_{\left(\binom{n+1}{2}-m-\binom{r+1}{2}\right)}(S^2\fU)\Seg_{\left(m-\binom{n-r+1}{2}\right)}(S^2\fQ^*).
	\]
	We then obtain the theorem by expanding these Segre classes, and using the fundamental duality properties of Schubert classes. 
\end{proof}
\begin{rem}
	Recall that our definition of $\psi_I$ is shifted w.r.t.\ \cite{BothmerRanestad}, which explains why our formula looks slightly different.
\end{rem}
\begin{rem} \label{rem:deltaExtend}
	One can easily verify that the above formula is still true for the extended definition of $\delta$ from \ref{rem:defExtension}. The only nontrivial case is $\delta\left(\binom{n+1}{2},n,0\right)=\psi_{[n]} \psi_{[n] \setminus [n]} =1$.
\end{rem}

\subsection{Representation theory}
In this subsection we will establish a formula which expresses the ML-degree as a linear combination of dimensions of irreducible representations of $\SL_n$. Our construction is based on the following folklore 
lemma.

\begin{lemma}\label{lem:prodink}
Let $X$ be a smooth complete $N$-dimensional algebraic variety, and $D_1, D_2$  two divisors on $X$.  
Then  the following identity holds:
$$
\int_XD_1^iD_2^{N-i} = \chi\left((1-\mathcal{O}(-D_1))^i(1-\mathcal{O}(-D_2))^{N-i}\right),
$$
where $\chi$ denotes the holomorphic Euler characteristic. 
\end{lemma}
\begin{proof}
By the additivity and multiplicativity properties of the Chern character we have
$$
\ch(1-\cO(-D_i)) = \sum_{k\geq 1} (-1)^{k+1} \frac{D_i^k}{k!},
$$
and therefore 
$$
\ch\!\left((1-\mathcal{O}(-D_1))^i(1-\mathcal{O}(-D_2))^{N-i}\right)=\! \left(\sum_{k\geq 1} (-1)^{k+1} \frac{D_1^k}{k!}\right)^i\!\!\!\left(\sum_{k\geq 1} (-1)^{k+1} \frac{D_2^k}{k!}\right)^{N-i} \!\!\!\!\!\!=
\int_XD_1^iD_2^{N-i}.
$$
Finally, by the Riemann-Roch theorem we get
$$
 \chi\left((1-\mathcal{O}(-D_1))^i(1-\mathcal{O}(-D_2))^{N-i}\right) = \int_X D_1^iD_2^{N-i}\td(X) = 
\int_X D_1^iD_2^{N-i}.  
$$
\end{proof}

We are going to apply Lemma~\ref{lem:prodink} to the computation of the ML-degree 
$$\phi(n,d) = \int_{CQ_n}L_1^{N+1-d} L_{n-1}^{d-1},$$ 
where $N=\binom{n+1}{2}-1$ denotes the dimension of the variety of complete quadrics $CQ_n$. 
We will denote by $\Lambda$ the character lattice of $\SL_{n}$, and for a fundamental  $\lambda\in \Lambda$, we will denote by $V_\lambda$ 
the corresponding irreducible representation of $\SL_{n}$. We will also denote by $\alpha_1,\ldots, \alpha_{n-1}$ and $\omega_1,\ldots, \omega_{n-1}$ 
the simple roots and fundamental weights of $\SL_{n-1}$, respectively. 

For a character $\lambda \in \Lambda$, let us define a $\SL_{n}$-representation $W_\lambda$ by
$$
W_\lambda := \bigoplus_{\nu\in \Lambda} V_{2\nu}^*,
$$
where the sum is taken over dominant weights of the form $\nu = \lambda - \sum_{i=1}^{n-1} k_i \alpha_i$, with $k_i\in \ZZ_{\geq 0}$. 
In particular, $W_\lambda=0$ if $\lambda$ cannot be represented as a sum $\nu+\sum_{i=1}^{n-1} k_i \alpha_i$, with $\nu$ a dominant weight and $k_i\in \ZZ_{\geq 0}$.
For $\lambda = (i-1)\omega_1 + (j-1)\omega_2 - \sum_{l=1}^{n-1}\omega_l$, we will denote the representation $W_\lambda$ by $W_{i,j}^n$.
\begin{thm}\label{thm:MLrep}
With the notation as above, the following identity holds:
$$
\phi(n,d) = 1+ \sum_{\substack{0\leq i\leq N-d+1\\
                  0\leq j\leq d-1\\
                  i+j>0}}
(-1)^{i+j+N}\binom{N+1-d}{i}\binom{d-1}{j}  \dim (W^n_{i,j}).
$$
\end{thm}

\begin{proof}
By Lemma~\ref{lem:prodink} we have:
\begin{multline*}
  \phi(n,d)=
\chi((1-\cL_1^{-1})^{N+1-d}(1-\cL_{n-1}^{-1})^{d-1}) =\\
\sum_{\substack{0\leq i\leq N+1-d\\
                  0\leq j\leq d-1}}
(-1)^{i+j}\binom{N+1-d}{i}\binom{d-1}{j} \chi(\cL_1^{-i}\otimes\cL_{n-1}^{-j}).
\end{multline*}
Now, both $\cL_1$ and $\cL_{n-1}$ are globally generated, and since they are pull-backs of ample line bundles by birational morphisms their Iitaka dimensions $\kappa(\cL_1)=\kappa(\cL_{n-1})=N$. By 
\cite[Theorem 2.2]{brion1990}, the cohomology of their negative powers must therefore vanish in 
degree lower than the dimension, so that for $i\geq 0, j\geq0$ and $i+j>0$
$$
\chi(\cL_1^{-i}\otimes\cL_{n-1}^{-j})=(-1)^{N} h^N(\CQ_n, \cL_1^{-i}\otimes\cL_{n-1}^{-j}) = 
(-1)^{N} h^0(CQ_n, K_{\CQ_n} \otimes \cL_1^{i}\otimes\cL_{n-1}^{j}).
$$
The canonical divisor $K_{\CQ_n}$ of the space of complete quadrics is given by (\cite[Corollary 3]{lozano2015})
$$K_{\CQ_n}= -L_1-L_{n-1}-\sum_{l=1}^{n-1}L_l.$$
In particular, $\CQ_n$ is Fano \cite{Massarenti1} and $\chi(\CQ_n,\cO)=1$. Finally, the spaces of global sections of line bundles on complete quadrics have been computed,
as $\SL_n$-representations, by De Concini and Procesi. It is direct corollary of \cite[Theorem 8.3]{DeConciniProcesi1} that
$$
H^0(\CQ_n, \cL_1^{a_1}\otimes\ldots\otimes \cL_{n-1}^{a_{n-1}}) = W_{a_1\omega_1+\ldots+a_{n-1}\omega_{n-1}}.
$$
In particular, $H^0(CQ_n, K_{\CQ_n} \otimes \cL_1^{i}\otimes\cL_{n-1}^{j})=W^n_{ij}$, so our result follows. 
\end{proof}

\begin{ex}\label{ex:4}
Let us  compute $\phi(3,1), \phi(3,2)$ and $\phi(3,3)$ with the help of Theorem~\ref{thm:MLrep}. First we notice that:
$$
W^3_{i,0} = W^3_{0,i}  = 0 \text{ for } i \leq 5, \quad W^3_{i,1} = W^3_{1,i} =  0 \text{ for } i \leq 3.
$$
Therefore we get:
\begin{equation*}\label{eq:phi(3,1)}
\phi(3,1) =  1 + \sum_{\substack{1\leq i\leq 5}}
(-1)^{i+1}\binom{5}{i} \dim\left(W^3_{i,0}\right) = 1;
\end{equation*}
\begin{equation*}\label{eq:phi(3,2)}
\phi(3,2) =  1+ \sum_{\substack{0\leq i\leq 4\\
                  0\leq j\leq 1\\
                  i+j>0}}
(-1)^{i+j+1}\binom{4}{i} \dim\left(W^3_{i,j}\right) = 1 + \dim(W^3_{4,1}) = 1 + \dim V_0 = 2;
\end{equation*}
\begin{equation*}\label{eq:phi(3,3)}
\phi(3,3) =  1+ \sum_{\substack{0\leq i\leq 3\\
                  0\leq j\leq 2\\
                  i+j>0}}
(-1)^{i+j+1}\binom{3}{i}\binom{2}{j} \dim\left(W^3_{i,j}\right) = 1 - 3\cdot \dim(W^3_{2,2}) +\dim(W^3_{3,2}),
\end{equation*}
where $W^3_{2,2}= V_0$  and  $W^3_{3,2} = V_{2\omega_1}^*$. By Weyl's dimension formula: 
$$
\dim V_{i\omega_1+j\omega_2}= \frac{(i+1)(j+1)(i+j+2)}{2},
$$
we get $\phi(3,3) = 1 - 3\cdot 1 + 6 = 4$.
\end{ex}
\begin{rem}
Our representation theoretic approach gives a closed formula for the ML-degree. However, already for $n=4$, the computation analogous to Example \ref{ex:4} is quite involved. One reason why this computation is more complicated than other formulas is that we obtain the answer to our intersection problem as a virtual representation, not only its dimension. One could say that this gives more information than we ask for. 
\end{rem}

\section{Polynomiality of the ML-degree}\label{sec:pol}
In this section and the next one, we present three proofs of the following polynomiality result for the algebraic degree of semidefinite programming:
\begin{thm}\label{thm:polyDelta}
	For any fixed $m,s >0$, the function $\delta(m,n,n-s)$ is a polynomial in $n$. Moreover this polynomial  vanishes at $n=0$. %
\end{thm}
As an immediate corollary, we obtain one of the main results of this paper: the polynomiality of the ML-degree for linear concentration models. This property was first conjectured by Sturmfels and Uhler \cite{StUh} and confirmed in small, special cases in \cite{MR3406441,MR1161918, MMW2020}.
\begin{thm}\label{thm:polyPHI}
For any fixed $d>0$, the function $\phi(n,d)$ is a polynomial for $n>0$.  
\end{thm}
\begin{proof}
For all $n,d>0$, by Corollary \ref{FundamentalRelation}, we have:
	\begin{equation} \label{eq:phideltapoly}
	\phi(n,d)=\frac{1}{n}\sum_{1 \leq \binom{s+1}{2} \leq d}{s\delta(d,n,n-s)}. 
	\end{equation}
	
	By \Cref{thm:polyDelta} every term in the right hand side of (\ref{eq:phideltapoly}) is a polynomial divisible by $n$, hence the theorem follows.
\end{proof}
Each of our proofs of \Cref{thm:polyDelta} has its advantages. The first one is quite elementary, 
being based on algebraic recursive formulas, which also have a geometric meaning. It provides very efficient methods for explicit computations. The second proof is more technical, however it allows to derive the leading coefficients of the polynomials we study. The last one is simply a corollary of the conjecture of Nie, Ranestad and Sturmfels that we prove in Section \ref{sec:NRS}.

Our first two proofs of \Cref{thm:polyDelta} are based on the following theorem. 
\begin{thm}\label{thm:polyPSI}
Let $I = \{i_1,\ldots,i_r\}$ be a set of strictly increasing nonnegative integers.
For $n\geq 0$ the function:
\[
LP_I(n):=
\begin{cases}
	\psi_{[n] \setminus I} &\text{ if } I \subseteq [n], \\
	0 &\text{ otherwise}.
	\end{cases}
	\]
 is a polynomial.
\end{thm}
Before we prove Theorem \ref{thm:polyPSI} let us note that it immediately implies Theorem \ref{thm:polyDelta}. Indeed, by \Cref{thm:delta}, we have
 \[
 \delta(m,n,n-s) = \sum_{\substack{I \subset [n] \\ \# I=s \\ \sum{I}=m-s}}{\psi_I \psi_{[n] \setminus I}} = \sum_{\substack{\# I=s \\ \sum{I}=m-s}}{\psi_I LP_I(n)} 
 \]
By Theorem \ref{thm:polyPSI}, each of the summands is a polynomial in $n$ that vanishes for $n=0$. Thus $\delta(m,n,n-s)$ is also a polynomial in $n$, which proves Theorem \ref{thm:polyDelta}, and hence Theorem \ref{thm:polyPHI}.

In the remainder of this section, we will present two proofs of Theorem \ref{thm:polyPSI}. 
But let us first give a few examples. 

\begin{ex}\label{ex:LP}
By induction, one can check the following formulas for $LP_I$, when $I$ has cardinality one or two:
$$LP_{(i)}(n)=\binom{n}{j+1}, \qquad LP_{(0,j)}(n)=j\binom{n+1}{j+2},$$
and more generally, for $i<j$, 
$$LP_{(i,j)}(n)=\frac{(j-i)[n+1]_{j+2}}{(i+1)!(j+1)!(i+j+2)!}\sum_{d=0}^{i}(-1)^da_{i,d}(i+j+1-d)![n]_{i-d},$$
where $a_{i,d}=\prod_{k=0}^{d-1}(i-k)(i-k+1)$ and $[n]_d=n(n-1)\cdots (n-d+1)$.
\end{ex}

\subsection{First proof}
The following recursive relations are central for our first proof.
\begin{lemma}\label{lem:recurssionsPSI}
\begin{enumerate}
\item For $j_1>0$ we have:
\begin{equation}\label{eq:rec1}
\psi_{\{j_1,\ldots,j_r\}}=(r+1)\psi_{\{0,j_1,\ldots,j_r\}}-2\sum_{\ell=1}^{r}{\psi_{\{0,j_1,\ldots,j_\ell-1,\ldots,j_r\}}},
\end{equation}
where  the summation is over all $\ell$ for which $j_{\ell}-1>j_{\ell-1}$ and we set $j_0:=0$.
\item For $j_1=0$ we have: \begin{equation}\label{eq:rec2}
\psi_{\{j_1,j_2,\ldots,j_r\}}=\sum_{j_\ell \leq j'_\ell < j_{\ell+1}}{\psi_{\{j'_1,\ldots,j'_{r-1}\}}}.
\end{equation}
\end{enumerate}
\end{lemma}
\begin{proof}
The first formula is: \cite[p.~446]{pragacz1988enumerative}, \cite[(A.15.7)]{LLT} and \cite[p.~163-166]{MR1481485}.

To prove the second formula, recall that $s_{(d)}$ is the complete homogeneous symmetric polynomial of degree $d$, 
and that we have:
\[
s_{(d)}(\{x_i + x_j \mid 1 \leq i \leq j \leq r\})= \sum_{\tiny \begin{matrix} \lambda(I) \vdash d \\ \# I= r \end{matrix}}\psi_I s_{\lambda(I)}(x_1,\ldots,x_r).
\]
Substituting $x_r=0$ we obtain:
\[
\sum_{i=0}^d s_{(i)}(\{x_i + x_j \mid 1 \leq i \leq j \leq r-1\})s_{(d-i)}(x_1,\dots,x_{r-1})=\]\[ s_{(d)}(\{x_i + x_j \mid 1 \leq i \leq j \leq r-1\},x_1,\dots,x_{r-1})=
\sum_{\tiny \begin{matrix} \lambda(I) \vdash d \\ \length(\lambda(I)) \leq r-1 \end{matrix}}\psi_I s_{\lambda(I)}(x_1,\ldots,x_{r-1}).\]
We note that $\length(\lambda(I)) \leq r-1$ if and only if $0\in I$. On the other hand we may apply Pieri's rule to $$\sum_{i=0}^d s_{(i)}(\{x_i + x_j \mid 1 \leq i \leq j \leq r-1\})s_{(d-i)}(x_1,\dots,x_{r-1})=$$
$$ \sum_{i=0}^d\left(\sum_{\tiny \begin{matrix} \lambda(I) \vdash i \\ \# I= r-1 \end{matrix}}\psi_I s_{\lambda(I)}(x_1,\ldots,x_{r-1})\right)s_{(d-i)}(x_1,\dots,x_{r-1}).$$
Comparing the coefficients of Schur polynomials in both expressions gives the formula. 
\end{proof}
\begin{proof}[First proof of Theorem \ref{thm:polyPSI}]

We proceed by induction first on  $\# I$, then on $\sum{I} :=\sum_{i_j\in I} i_j$. 
The base case is $I=\emptyset$, when  $\psi_{\{0,\ldots,n-1\}}=1$.
 
	For the induction step, fix $I$, and assume the theorem has been proven for all $I'$ with $\# I' < \# I$, and for all $I'$ with $\# I'=\# I$ and $\sum{I'}<\sum{I}$. We consider two cases:
	
	{\textbf{Case 1.}} $i_1 =0 $. We claim that for every $n\geq 0$, 
	\[
	LP_I(n)=(n-r+1)LP_{I \setminus \{0\}}(n)-2\sum_{\ell: i_{\ell+1}>i_\ell+1}{LP_{I \setminus \{0,i_\ell\} \sqcup \{i_\ell+1\}}(n)},
	\]
	where for summation we formally assume $i_{r+1}= + \infty$.
	Indeed: if $n\leq i_r$ then both sides are $0$, and if $n > i_r$ then the equation is precisely Lemma \ref{lem:recurssionsPSI} (1).
	
	{\textbf{Case 2.}} $i_1 > 0 $. We claim that for every $n\geq 0$, 
	\[
	LP_I(n)-LP_I(n-1)=\sum_J{LP_J(n-1)},
	\]
	where the sum is over all $J \neq I$ of the form $\{i_1-\epsilon_1, \ldots, i_r-\epsilon_r\}$ with $\epsilon_\ell \in \{0,1\}$.
	Again, if $n\leq i_r$ then both sides are $0$, and if $n > i_r$ then the equation is precisely Lemma \ref{lem:recurssionsPSI} (2).
	
	In both cases, it follows that $LP_I$ is a polynomial.
\end{proof}

\subsection{Second proof}
Our second proof is based on an explicit interpretation of $\psi_I$ as a sum of minors in the Pascal triangle. We denote by $E$ the Pascal triangle matrix, i.e. $E_{ij}=\binom ij$. We will always consider only finite submatrices of $E$ so despite the fact that it is an infinite matrix there will be no computations with infinite matrices.

\begin{notation}
For an $n\times n$ matrix $A$ and sets $I,J\subset [n]$ we denote by $A_{I,J}$ the $\# I\times \#J$ matrix which is obtained from $A$ by taking rows indexed by $I$ and columns indexed by $J$. Here we index rows and columns from 0. In the case $I=J$ we write simply $A_{I,I}=A_I$.
	
For sets $K,C \subset \N$ with $\# K=\# C$ we denote $V(K,C)$ the Vandermonde matrix with entries $V(K,C)_{ij}=k_{i+1}^{c_{j+1}}$. We also set $V(K):=V(K,[\# K])$, i.e.~$V(K)_{ij}=k_{i+1}^j$.

For two sets $A,B \subset \N$ we denote by $\varepsilon^{A,B}$ the sign of the permutation of $A\cup B$ determined by $A,B$ if they are disjoint. If they are not, we define $\varepsilon^{A,B}=0$. 
\end{notation}

We begin with a characterization of $\psi_I$ as a sum of the minors of the matrix $E$ which follows from \cite[Proposition 2.8]{LLT}.
\begin{prop}\label{psipascal} The following equality holds:
\[\psi_I=\sum_{J\le I}\det (E_{I,J}).\]
\end{prop}
In what follows we will need the following lemma that may be easily proved by induction.
\begin{lemma}\label{pascalidentity}
Let $a,b$ be nonnegative integers. 
\begin{itemize}
\item[a)] If $a>b$ then $\sum_{i=0}^{a} (-1)^i{a\choose i}i^b=0$.
\item[b)] If $a=b$ then $\sum_{i=0}^a (-1)^{a-i}{a\choose i}i^b=a!$. 
\end{itemize}
\end{lemma}

To compute special minors of the matrix $E$ we use the following lemma.

\begin{lemma}\label{bigminor}
Let $I=\{i_1,\dots,i_r\}\subset[n]$ be a set of nonnegative integers. Then \[\det E_{[n]\setminus [r],[n]\setminus I}=\frac{\prod_{1\le j<k\le n-r}{(i_k-i_j)}}{(r-1)!(r-2)!\dots 2!1!}=\frac{\det (V(I))}{(r-1)!(r-2)!\dots 2!1!}\]

\end{lemma}

\begin{proof}
We fix $r$ and proceed by induction on $n$. The case $i_r<n-1$ is trivial. In the case $i_r=n-1$ we express the determinant via Laplace expansion on the $n$-th row, use the induction hypothesis and Lemma \ref{pascalidentity} to conclude. 
\end{proof}

Now we are able to present our second proof of Theorem \ref{thm:polyPSI}. 

\begin{proof}[Second proof of Theorem \ref{thm:polyPSI}]
Let $\# I=r$ and $m:=i_r+1$.  First, assume $n\geq m$. We use the formula from Proposition \ref{psipascal}. We express the determinants $E_{[n]\setminus I,J}$ using the Laplace expansion along the first $m-r$ rows, we choose the columns indexed by set $L$. For the rest we use the Lemma \ref{bigminor}. To simplify notation we let $K:=[n]\setminus J$. 

\begin{align*}
\psi_{[n]\setminus I}&=\sum_{J\le [n]\setminus I}\det (E_{[n]\setminus I,J})\\
&=\sum_{J\le [n]\setminus I}\sum_{\substack{L\subseteq J\\ \# L=m-r}}\varepsilon^{L,J\setminus L}\det (E_{[m]\setminus I,L})\det(E_{[n]\setminus[m],J\setminus L})\\
&=\sum_{\substack{\# L=m-r\\ L\le [m]\setminus I}}\det(E_{[m]\setminus I, L})\sum_{\substack{\# K=r\\ K\cap L=\emptyset\\ K\subset [n]}}\varepsilon^{L,[n]\setminus(K\cup L)}\det(E_{[n]\setminus[m],[n]\setminus (K \cup L)})\\
&=\sum_{\substack{\# L=m-r\\ L\le [m]\setminus I}}\det(E_{[m]\setminus I, L})\sum_{\substack{\# K=r\\  K\subset [n]}}\varepsilon^{L,K}\varepsilon^{L,[n]\setminus L}\frac {\det (V(L\cup K))}{(m-1)!(m-2)!\dots 2!1!} \\
&=\sum_{\substack{\# L=m-r\\ L\le [m]\setminus I}}\frac{\varepsilon^{L,[n]\setminus L}\det(E_{[m]\setminus I, L})}{(m-1)!(m-2)!\dots 2!1!} \sum_{\substack{\# K=r\\  K\subset [n]}}\det (V^*(L\cup K))
\end{align*}
where $V^*(L\cup K)$ is the matrix $V(L\cup K)$ where we first put the rows indexed by $L$. Note that we may drop the assumption $L\le [m]\setminus I$, since otherwise $\det(E_{[m]\setminus I, L})=0$. Similarly, we can extend our sum and drop the condition $L\cap K=\emptyset$ since we add only zero terms. If we fix $L$ and denote the elements of $K$ by $k_1<\dots<k_r$, then $\det (V^*(L\cup K))$ is clearly a polynomial in $k_1,\dots,k_r$. Then 
\[\sum_{\substack{\# K=r\\ K\subset [n]}} \det (V^*(L\cup K))=\sum_{0\le k_1<\dots<k_r<n} \det (V^*(L\cup K))\]
is a polynomial in $n$ for the fixed $L$. Moreover, the sum through $L$ does not depend on $n$ and therefore also $\psi_{[n]\setminus I}$ is a polynomial in $n$. Our computations are correct only for $n\ge m$. However, the last expression makes sense and is a polynomial for all $n\ge 0$. Clearly, it is equal $0$ for $n<m$. This proves the theorem. 
\end{proof}

With this approach we can even compute the leading coefficient of $LP_I$. For this we will need two technical lemmas. The proof of the first one is straightforward, e.g.~by induction.

\begin{lemma}\label{sumtroughn}
Let $a_1,\dots,a_r$ be nonnegative integers. Then
\[\sum_{0\le k_1<\dots<k_r<n} k_1^{a_1}k_2^{a_2}\dots k_r^{a_r}\]
is a polynomial in $n$ of degree $\sum_{i=1}^r a_i+r$,
with leading coefficient  $\frac{1}{(a_1+1)(a_1+a_2+2)\dots (a_1+\dots+a_r+r)}$.
\end{lemma}

\begin{lemma}\label{horriblefractionsum}
	The following identity of rational functions in $r$ variables holds:
\[\sum_{\sigma\in S_r}(-1)^\sigma \frac{1}{
(x_{\sigma(1)})(x_{\sigma(1)}+x_{\sigma(2)})\dots(x_{\sigma(1)}+\dots+x_{\sigma(r)})}=\frac{\prod_{i>j}(x_i-x_j)}{\prod_i x_i \prod_{i>j}(x_i+x_j)}\]
\end{lemma}
\begin{proof}
We proceed by induction on $r$. It is easy to check that for $r=1,2$ the statement holds.
For $r>2$, we split the sum depending on $\sigma(r)$ and apply the induction hypothesis to the partial sums:
\begin{align*} 
&\sum_{\sigma\in  S_r}(-1)^\sigma \frac{1}{
(x_{\sigma(1)})(x_{\sigma(1)}+x_{\sigma(2)})\dots(x_{\sigma(1)}+\dots+x_{\sigma(r)})}=\\
&=\frac{1}{x_{1}+\dots+x_{r}}\sum_{k=1}^{r}\sum_{\substack{\sigma\in S_{r}\\ \sigma(r)=k}}(-1)^\sigma \frac{1}{
(x_{\sigma(1)})(x_{\sigma(1)}+x_{\sigma(2)})\dots(x_{\sigma(1)}+\dots+x_{\sigma(r-1)})}\\
&=\frac{1}{x_{1}+\dots+x_{r}}\sum_{k=1}^{r} (-1)^{r-k}\frac{\prod_{i>j;i,j\neq k}(x_i-x_j)}{\prod_{i\neq k} x_i \prod_{i>j;i,j,\neq k}(x_i+x_j)}\\
&=\frac{1}{(x_{1}+\dots+x_{r})\prod_i x_i \prod_{i>j}(x_i+x_j)}\sum_{k=1}^{r} (-1)^{r-k} x_k\prod_{i>j;i,j\neq k}(x_i-x_j)\prod_{i\neq k} (x_i+x_k)\\
&=\frac{1}{(x_{1}+\dots+x_{r})\prod_i x_i \prod_{i>j}(x_i+x_j)} Q(x_1,\dots,x_r),
\end{align*}
where $Q$ is a homogeneous polynomial of degree ${r\choose 2}+1$. Moreover, $Q$ is skew-symmetric, that if we exchange values of $x_i$ and $x_j$ we just change the sign. Therefore \[Q(x_1,\dots,x_r)=\prod_{i>j}(x_i-x_j)R(x_1,\dots,x_r)\]
for $R$ a symmetric polynomial of degree one. This implies that $R$ is a multiple of $x_1+\cdots+x_r$. 
Finally, it is easy to check that the coefficient of $x_r^r x_{r-1}^{r-2}x_{r-2}^{r-3}\dots x_2$ in $Q$ is $1$.  Therefore $R=x_1+\dots+x_r$ and the proof is complete. 
\end{proof}

\begin{thm} \label{thm:leadingTerm}
The polynomial $LP_I$ has degree $\sum I+\# I$. Its leading coefficient is equal to
\[\frac{\prod_{j>k}(i_j-i_k)}{(i_1+1)!\dots(i_r+1)! \prod_{j>k}(i_j+i_k+2)}\]
\end{thm}
  
\begin{proof}
We continue with the calculation from the second proof of Theorem \ref{thm:polyPSI}. We do Laplace expansion of Vandermonde by first $m-r$ rows. We get
\begin{align*}
&\sum_{\substack{\# L=m-r\\ L\le [m]\setminus I}}\varepsilon^{L,[m]\setminus L}\det(E_{[m]\setminus I, L})\sum_{\substack{\# K=r\\ K\subset [n]}} \det (V^*(L\cup K))=\\
&=\sum_{\substack{\# L=m-r\\ L\le [m]\setminus I}}\varepsilon^{L,[m]\setminus L}\det(E_{[m]\setminus I, L})\sum_{\substack{\# C=m-r\\ C\subset [m]}}\varepsilon^{C,[m]\setminus C} \det (V(L,C)) \sum_{\substack{\# K=r\\ K\subset [n]}} \det(V(K,[m]\setminus C))\\
&=\sum_{\substack{\# C=m-r\\ C\subset [m]}}\varepsilon^{C,[m]\setminus C}\sum_{\substack{\# L=m-r\\ L\le [m]\setminus I}}\varepsilon^{L,[m]\setminus L} \det(E_{[m]\setminus I, L})\det (V(L,C))\sum_{\substack{\# K=r\\ K\subset [n]}} \det(V(K,[m]\setminus C))\\
&=\sum_{\substack{\# C=m-r\\ C\subset [m]}}\varepsilon^{C,[m]\setminus C}\det \left(\text{diag}(1,-1,\dots,(-1)^{m-1})E_{[m]\setminus I, [m]}V([m],C)\right)\sum_{\substack{\# K=r\\ K\subset [n]}} \det(V(K,[m]\setminus C)).
\end{align*}
Consider the matrix $A:=\left(\text{diag}(1,-1,\dots,-1^{m-1})E_{[m]\setminus I, [m]}V([m],C)\right)$. Let $[m]\setminus I=\{b_1,\dots,b_{m-r}\}$, $C=\{c_1,\dots,c_{m-r}\}$, where, as always, we assume that the elements of these sets are ordered increasingly. Notice that $c_{m-r}<b_{m-r}$ implies that the last row of the matrix $A$ is 0 by Lemma \ref{pascalidentity} and so is $\det(A)$. In general, if $c_{i}<b_{i}$, then $A_{[m-r]\setminus[i-1],[i]}=0$ and we also get $\det A=0$. 

The necessary condition for $\det A\neq 0$ is $c_i\ge b_i$ for all $1\le i\le m-r$. Therefore, we will sum only through such sets $C$. In the border case when $C=[m]\setminus I$ we get that the matrix $A$ is upper triangular and by Lemma \ref{pascalidentity} we have $\varepsilon^{C,[m]\setminus C}\det A=(b_1)!\dots (b_{m-r})!$. 

The sum $\sum_{\substack{\# K=r\\ K\subset [n]}} \det(V(K,[m]\setminus C))$
 is clearly a polynomial in $n$ of degree at most $\sum ([m]\setminus C) + r={m\choose 2}+r-
\sum C$. Since we are summing only trough $C$ with $\sum C\ge \sum ([m]\setminus I)$ we immediately get that the degree of the polynomial $P$ is at most $\sum I+r$. 
Moreover, the only summand which contributes to the term of degree ${\sum I+r}$ is the one with $C=[m]\setminus I$. We finish the proof of the theorem by computing this summand. In this case we get the polynomial

\[\widetilde{LP}_I(n):=\sum_{\substack{\# K=r\\ K\subset [n]}} \det(V(K,[m]\setminus C))=\sum_{\sigma\in S_r}\sum _ {0\le k_1<\dots<k_r<n}(-1)^\sigma k_1^{i_{\sigma(1)}}\dots k_r^{i_{\sigma(r)}}.\] 

By Lemma \ref{sumtroughn} the leading coefficient of $\widetilde{LP}_I$ is
\[\sum_{\sigma\in \mathbb S_r}(-1)^\sigma \frac{1}{
(i_{\sigma(1)}+1)(i_{\sigma(1)}+i_{\sigma(2)}+2)\dots(i_{\sigma(1)}+\dots+i_{\sigma(r)}+r)}\]
 Now we apply Lemma \ref{horriblefractionsum} for $x_j=i_j+1$ to conclude that the leading coefficient of $\widetilde{LP}_I$ is  
\[\frac{\prod_{j>k}(i_j-i_k)}{\prod_j (i_j+1) \prod_{j>k}(i_j+i_k+2)}\]
which is obviously non-zero. This shows that the degree of the polynomial $LP_I$ is $\sum I+r$ and its leading coeffient is 

\[\frac{1}{(m-1)!(m-2)!\dots 1!}\cdot (b_1!)\dots (b_{m-r})!\cdot\frac{\prod_{j>k}(i_j-i_k)}{\prod_j (i_j+1) \prod_{j>k}(i_j+i_k+2)}=\]
\[=\frac{\prod_{j>k}(i_j-i_k)}{(i_1)!\dots(i_r)!\prod_j (i_j+1) \prod_{j>k}(i_j+i_k+2)}.\]

\end{proof}

\begin{cor}
	The polynomial $\delta(m,n,n-s)$ from \Cref{thm:polyDelta} has degree $m$, and the polynomial $\phi(n,d)$ from \Cref{thm:polyPHI} has degree $d-1$.
\end{cor}

\section{The Nie-Ranestad-Sturmfels conjecture}\label{sec:NRS}
In this section we present a proof of the formula for the degree of semidefinite programming which was conjectured by Nie, Ranestad and Sturmfels \cite{NRS}. The formula was known so far only for special values of the parameters. To state it we introduce the following coefficients.
\begin{defi}[Coefficients $b_I$]
Let $I$ be a set of $s$ nonnegative integers. We define $b_I(n)$ by the following formula:
\[b_I(n)=Q_{I+\mathbf{1}_s}(\underbrace{1/2,\dots, 1/2}_{n\;\mathrm{times}}),\]
where $I+\mathbf{1}_s$ is the set obtained from $I$ by adding one to each of its elements. The function $Q_{I+\mathbf{1}_s}$ is the Schur $Q$-function \cite[Section III.8]{Macdonald} and its argument 
$1/2$ appears $n$ times.

These coefficients may be computed recursively as described in \cite[Section 6]{NRS}. We note that in this reference the authors use the convention that $I$ is a subset of the set $\{1,\dots,n\}$ while in this article $I\subset [n]=\{0,\dots,n-1\}$. This results in the difference in notation for the coefficient $b_I$ exchanging $I$ and $I+\mathbf{1}_s$. 
\end{defi}

The main theorem of this section, confirming the Nie-Ranestad-Sturmfels conjecture, is the following.

\begin{thm}\label{nrs} (\cite{NRS}, Conjecture 21)
Let $m,n,s$ be positive integers. Then
\[\delta(m,n,n-s)=\sum_{\sum I\le m-s} (-1)^{m-s-\sum I}\psi_I b_I(n)\binom {m-1} {m-s-\sum I}.\]
where the sum goes through all sets of nonnegative integers of cardinality $s$. 

\end{thm}
As we already mentioned, \Cref{thm:polyDelta} is an immediate corollary of \Cref{nrs}, since the coefficients $b_I(n)$ are known to be polynomials. Hence, as soon as we have proven \Cref{nrs}, we have a third proof of \Cref{thm:polyDelta}.
\begin{rem}
We note that if the Pataki inequality \eqref{eq:Pataki} $m\ge \binom {s+1}2$ is not satisfied, then both sides of the equality above are trivially zero. 
\end{rem}

For the rest of the section we fix the numbers $m,n,s$ as in the statement of the theorem. %
Theorem \ref{nrs} presents a relation between numbers $b_I(n)$ and $\psi_I$, our proof of which will be algebraic, with the coefficients $s_{I,J}$ from Definition \ref{def:sIJ} playing a prominent role. The following lemma describes the relations between $b_I(n)$ and $s_{I,J}$:

\begin{lemma}\label{lemmab_I}
Let $I$ be a set of $s$ nonnegative integers. Then 
\[b_I(n)=\sum_{J\le I} \left(\frac{1}{2}\right)^{\sum I-\sum J} s_{I,J} LP_J(n),\]
\[LP_I(n)=\sum_{J\le I} \left(-\frac{1}{2}\right)^{\sum I-\sum J} s_{I,J} b_J(n).\]
\end{lemma}

These two identities are equivalent, by the discussion following Definition \ref{def:sIJ}. 
We present two proofs: one based on simple algebra, and one on more sophisticated methods from algebraic geometry.

For the first proof, let us recall two statements from linear algebra which will allow us to prove Pfaffian formulas also for the set complements.

\begin{lemma}\label{jac} (Jacobi's Theorem.)
	Let $A$ be an $n\times n$ matrix, and $A^C$  its cofactor matrix. Then \[\det(A_{[n]\setminus I,[n]\setminus J})=\det (A^C_{I,J})\det(A)^{\#I-1}\]
	for all sets $I,J\subset [n]$ with $\#I=\#J$.
\end{lemma}

\begin{cor}\label{cofactor}
	The cofactor matrix $A^C$ of an $n\times n$ skew-symmetric matrix $A$ is given by
	$$A^C_{ij}=\Pf(A_{[n]\setminus{\{i,j\}}})\Pf(A).$$
\end{cor}

\begin{lemma}
	Let $I=\{i_1,\dots,i_r\}$ be a set of nonnegative integers. Then \[\psi_{[n]\setminus I}=\Pf(\psi_{ [n]\setminus {\{i_k,i_l\}}})_{0<k<l\le r} \text{ for even }\#I,\]
	\[\psi_{[n]\setminus I}=\Pf(\psi_{ [n]\setminus {\{i_k,i_l\}}})_{0\le k<l\le r} \text{ for odd }\#I,\]
	where $\psi_{ [n]\setminus {\{i_0,i_k\}}}:=\psi_{ [n]\setminus {\{i_k\}}}.$
\end{lemma}

\begin{proof}
	Let us consider the case where both $n$ and $\#I$ are even. Consider the skew-symmetric matrix $A$ such that $A_{k,l}=\psi_{\{k,l\}}$ for $0\le k<l<n$. Then using Lemmas \ref{jac} and \ref{pfaffianpsi} we get\[\psi_{[n]\setminus I}=\Pf(A_{[n]\setminus I})=\Pf(A^C_I)\Pf(A)^{\#I-1}=\Pf(A^C_I),\]
	since $\det(A)=\psi_{\{0,1,\dots,n-1\}}=1$. 
	Moreover, by Corollary~\ref{cofactor}, the entries of the cofactor matrix $A^C$ are $\Pf(A_{[n]\setminus{\{k,l\}}})\Pf(A)=\psi_{[n]\setminus{\{k,l\}}}$ which proves the lemma in this case.
	
	The proof in the other cases is similar. The only difference is that we consider a different matrix $A$. If $n$ is odd we take $A=(\psi_{\{k,l\}})_{-1\le k<l<n}$ and if $n$ is even and $\#I$ is odd we take $A=(\psi_{\{k,l\}})_{-2\le k<l<n}$. We interpret $\psi_{\{-1,k\}}$ and $\psi_{\{-2,k\}}$ as $\psi_{\{k\}}$ and we put $\psi_{\{-1,-2\}}=1$. Then we conclude in the same way.
\end{proof}

\begin{cor}\label{cor:Pfaffrecursion}
	\[
	\#I\psi_{[n]\setminus{I}}=\begin{cases}2\sum_{1 \leq k < l \leq r} {(-1)^{k+l+1}\psi_{[n]\setminus{\{i_k,i_l\}}} \psi_{[n]\setminus{(I \setminus \{i_k,i_l\})}}} & \text{ if $\#I$ is even} \\ 2\sum_{0 \leq k < l \leq r} {(-1)^{k+l+1}\psi_{[n]\setminus{\{i_k,i_l\}}} \psi_{[n]\setminus{(I \setminus \{i_k,i_l\})}}} & \text{ if $\#I$ is odd.}\end{cases}
	\]
	where $\psi_{ [n]\setminus {\{i_0,i_k\}}}:=\psi_{ [n]\setminus {\{i_k\}}}.$
\end{cor}
\begin{proof}
	For every skew-symmetric $r \times r$ matrix $A$ (with $r$ even) and every $k=1, \ldots, r$, we have the following recursive formula for the Pfaffian:
	\[
	\Pf(A)=\sum_{l=1}^{k-1}{(-1)^{k+l}a_{k,l}\Pf(A_{\hat{k}\hat{l}})} - \sum_{l=k+1}^{r}{(-1)^{k+l}a_{k,l}\Pf(A_{\hat{k}\hat{l}})},
	\]
	where $A_{\hat{k}\hat{l}}$ is the submatrix obtained by removing the $k$-th and $l$-th rows and columns.
	Summing over all $k$ gives the desired equality.
\end{proof}

\begin{rem}
	If we define $\psi_{[n]\setminus{I}}=0$ for $I=\{i_1,\ldots,i_r\}$ a multiset/partition with at least one repeated entry, the recursion from \Cref{cor:Pfaffrecursion} still holds. 
\end{rem}

\begin{rem}
	\Cref{cor:Pfaffrecursion} can be seen as a recursive relation between the polynomials $LP_I(n)$ from \Cref{thm:polyPSI}. In particular, we can obtain in this way one more proof of \Cref{thm:polyPSI}.
\end{rem}

\begin{proof}[First proof of \Cref{lemmab_I}]
	We will use induction on the length of $I$, which we will denote by $s$. The base of induction, i.e.~the cases $s=1,2$ are left for the reader.

	We proceed with the general case $s>2$. We will assume that $s$ is even; the odd case is analogous.
	Since $b_I=\Pf(b_{i_p,i_q})_{1 \leq p < q \leq s}$, we have (as in \Cref{cor:Pfaffrecursion}) the following recursive relations between the $b_I$'s:
	\[
	sb_I=2\sum_{1\leq p < q \leq n}{(-1)^{p+q+1}b_{\{i_p,i_q\}}b_{I \setminus \{i_p,i_q\}}}.
	\]
	In order to use induction, we need to show that
	\[
	s\sum_{J\le I}{ 2^{\sum{J}} s_{I,J} \psi_{[n]\setminus{J}}}=\hspace*{4cm}\]\[2\sum_{1\leq p < q \leq n}{(-1)^{p+q+1} \left(\sum_{J\le \{i_p,i_q\}} {2^{\sum{J}}  s_{\{i_p,i_q\},J} \psi_{[n]\setminus{J}}}\right) \left(\sum_{J\le I \setminus \{i_p,i_q\}} {2^{\sum{J}}  s_{I \setminus \{i_p,i_q\},J} \psi_{[n]\setminus{J}}}\right)}.
	\]
	This follows immediately from the following claim:
	\begin{claim}
		For every $J \leq I$, where $J$ can have repeated elements,
	\begin{equation*}%
	 {s_{I,J} \psi_{[n]\setminus{J}}}=\frac{2}{s}\sum_{1\leq p <  q \leq n}{(-1)^{p+q+1} \left(\sum_{1 \leq s < t < n} {s_{\{i_p,i_q\},\{j_s,j_t\}} \psi_{[n]\setminus{\{j_s,j_t\}}} s_{I \setminus \{i_p,i_q\},J \setminus \{j_s,j_t\}} \psi_{[n]\setminus{(J \setminus \{j_s,j_t\})}}} \right)}.
	\end{equation*}
	\end{claim}

	Indeed, using Laplace expansion, for any $t,u$ we can write:
	\[
	s_{I,J} = \sum_{p<q}{(-1)^{p+q+t+u}s_{\{i_p,i_q\},\{j_t,j_u\}}s_{I \setminus \{i_p,i_q\},J \setminus \{j_t,j_u\}}}.
	\]
	Hence, the right hand side can be rewritten as
	\[
	2s_{I,J}\sum_{1 \leq t < u < n} {(-1)^{t+u+1}\psi_{[n]\setminus{\{j_t,j_u\}}} \psi_{[n]\setminus{(J \setminus \{j_t,j_u\})}}}.
	\]
It remains to show that
	\[
	s\psi_{[n]\setminus{J}}=2\sum_{1 \leq t < u < n} {(-1)^{t+u+1}\psi_{[n]\setminus{\{j_t,j_u\}}} \psi_{[n]\setminus{(J \setminus \{j_t,j_u\})}}}.
	\]
But this is precisely \Cref{cor:Pfaffrecursion}, and this concludes the first proof of the formula.
\end{proof}
The ideas of the second proof were suggested to us by Andrzej Weber.
\begin{proof}[Second proof of \Cref{lemmab_I}]
We start with a projection formula, which is a special case of \cite[(4.7)]{Fulton_Pragacz}. The proofs of the formula were first provided by Pragacz \cite{MR1481485, MR3391029}.
Note that this formula is stated in terms of Schur $P$-polynomials, while we work with Schur $Q$-polynomials which accounts for an additional factor of a power of two.

For a vector bundle $\mathcal{E}$ of rank $n$ over some base $X$, we consider the relative Grassmannian
$G^s(\mathcal{E})$ of rank $s$ quotients of $\mathcal{E}$, with its projection $\pi$ to $X$. 
We denote by $\mathcal{K}$ and $\mathcal{Q}$ the relative tautological subbundle and quotient 
bundle of $\pi^*\mathcal{E}$, of respective ranks $r=n-s$ and $s$. Then 
\begin{equation}\label{projection}
Q_{I+\mathbf{1}_s}(\mathcal{E})=\pi_*(c_{top}(\mathcal{K}\otimes\mathcal{Q})Q_{I+\mathbf{1}_s}(\mathcal{Q})),
\end{equation}
where by $+\mathbf{1}_s$ we mean adding $1$ to all $s$ elements of $I$ (cf.~\cite[Example 2, p.~50]{Fulton_Pragacz}). 
Moreover,  \cite[(4.5)]{Fulton_Pragacz}, \cite[Proposition 2.2]{pragacz1991algebro} can be written as  
$$Q_{I+\mathbf{1}_s}(\mathcal{Q})=2^sc_{top}(\wedge^2\mathcal{Q})s_{\lambda(I)+\mathbf{1}_s}(\mathcal{Q})=
c_{top}(S^2\mathcal{Q})s_{\lambda(I)}(\mathcal{Q}).$$
Since $\pi^*\mathcal{E}$ is an extension of $\mathcal{Q}$ by $\mathcal{K}$, the bundle $\pi^*S^2\mathcal{E}$
admits a filtration whose successive quotients are $S^2\mathcal{Q}$, $\mathcal{K}\otimes\mathcal{Q}$ and 
$S^2\mathcal{K}$. Hence the identity 
$$c(\mathcal{K}\otimes\mathcal{Q})c(S^2\mathcal{Q})=s(S^2\mathcal{K})\pi^*c(S^2\mathcal{E}).$$
Equation (\ref{projection}) can thus be rewritten as
$$Q_{I+\mathbf{1}_s}(\mathcal{E})=c(S^2\mathcal{E})\pi_*(s(S^2\mathcal{K})s_{\lambda(I)}(\mathcal{Q}))_{|deg=\Sigma I+s},$$
where the last symbols mean we only keep the component of degree $\sum I+s$. 

Now suppose that $\mathcal{E}=\mathcal{E}_0\otimes L$ for some line bundle $L$ and a trivial vector bundle $\mathcal{E}_0$. 
Then $G^s(\mathcal{E})$ is a trivial bundle over $X$, while $\mathcal{K}=\mathcal{K}_0\otimes L$ and $\mathcal{Q}=\mathcal{Q}_0\otimes L$
are obtained by pull-back of the tautological and quotient bundles $\mathcal{K}_0$, $\mathcal{Q}_0$ over a fixed Grassmannian 
$G^s(\mathbf{C}^n)$ (we omit the pull-backs for simplicity). By Definition \ref{def:sIJ} (where formally 
the $x_i$'s are the Chern roots of $\mathcal{Q}$ and we need to homogenize by using $c_1(L)$), we have:
$$s_{\lambda(I)}(\mathcal{Q})=\sum_{J\le I}s_{I,J}s_{\lambda(J)}(\mathcal{Q}_0)\delta^{\Sigma I-\Sigma J},$$
where $\delta=c_1(L)$. Moreover, the Segre classes of $S^2\mathcal{K}_0^*$ and  $S^2\mathcal{K}$ 
are related by the formula 
$$s(S^2\mathcal{K})=\sum_{\ell\ge 0} (1+2\delta)^{-\binom{r+1}{2}-\ell}s_{(\ell)}(S^2\mathcal{K}_0^*).$$
Plugging these two formulas into the previous one, we get $Q_{I+\mathbf{1}_s}(\mathcal{E})$ as 
$$\sum_{J\le I}\sum_L(1+2\delta)^{\binom{n+1}{2}-\binom{r+1}{2}-|\lambda(L)|}\delta^{\Sigma I-\Sigma J}s_{I,J}\psi_L
\pi_*(s_{\lambda(L)}(\mathcal{K}_0^*)s_{\lambda(J)}(\mathcal{Q}_0))_{|deg=\Sigma I+s}.$$
Now recall that the Schur classes $s_\alpha(\mathcal{K}_0^*)$ and $s_\beta(\mathcal{Q}_0)$, 
for partitions $\alpha\subset (s^r)$ and $\beta\subset (s^k)$, that are non zero, give  dual bases of Schubert cycles on the 
Grassmannian $G^s(\mathbf{C}^n)$. This can be expressed as 
$$\pi_*(s_{\lambda(L)}(\mathcal{K}_0^*)s_{\lambda(J)}(\mathcal{Q}_0))=\delta_{L,[n]/J},$$
where $\delta_{L,[n]/J}$ is the Kronecker delta.
Note that $L=[n]/J$ implies that $|\lambda(L)|+|\lambda(J)|=sr$. We thus get the formula 
$$Q_{I+\mathbf{1}_s}(\mathcal{E})=\Big(\sum_{J\le I}
(1+2\delta)^{s+\Sigma J}\delta^{\Sigma I-\Sigma J}s_{I,J}\psi_{[n]/J}\Big)_{|deg=\Sigma I+s}.$$
But since the degree of the polynomial in brackets is exactly $\Sigma I+s$, we just need to keep its 
top degree component, that is 
$$Q_{I+\mathbf{1}_s}(\mathcal{E})=\sum_{J\le I}2^{\Sigma J+s}s_{I,J}\psi_{[n]/J} \delta^{\Sigma I+s}.$$
We conclude by applying formally this formula to the bundle $\mathcal{E}=\mathcal{O}(1/2)^{\oplus n}$ over the projective space. 
\end{proof}

\begin{lemma} \label{lemmasij}
Let $J$ be a set of nonnegative integers of length $s$ with $\sum J\le m-s$. Then 

\[\sum_{\substack{I\ge J\\ \sum I\le m-s}} \psi_I \left(-\frac{1}{2}\right)^{\sum I-\sum J}s_{I,J}\binom {m-1} {m-s-\sum I}=
\begin{cases} 0 &\text{ if } \sum J<m-s\\
\psi_J &\text{ if } \sum J=m-s

\end{cases}\]
\end{lemma}
\begin{proof}
We prove the lemma at the same time for all the $J$'s by multiplying the above equation by the Schur polynomial $s_{\lambda(J)}(x_1,\dots,x_s)$ and summing up. Since Schur polynomials form a basis of the space of symmetric polynomials, the statement of the lemma is equivalent to the following polynomial identity:
\[\sum_{\sum J\le m-s}\sum_{\substack{I\ge J\\ \sum I\le m-s}} \psi_I \left(-\frac{1}{2}\right)^{\sum I-\sum J}s_{I,J}\binom {m-1} {m-s-\sum I} s_{\lambda(J)}(x_1,\dots,x_s)=\]\[
\hspace*{3cm}=\sum_{\sum J=m-s} \psi_J s_{\lambda(J)}(x_1,\dots,x_s).\]

By \ref{def:psi}, the right hand side is equal to $s_{(m-s-\binom{s}{2})}(x_i+x_j\vert 1\le i\le j\le s)$. For the left hand side we can use Definition \ref{def:sIJ} of the coefficients $s_{I,J}$:
\begin{align*}
\sum_{\sum J\le m-s}\sum_{\substack{I\ge J\\ \sum I\le m-s}} \psi_I \left(-\frac{1}{2}\right)^{\sum I-\sum J}s_{I,J}\binom {m-1} {m-s-\sum I} s_{\lambda(J)}(x_1,\dots,x_s)&=\\
\sum_{\sum I\le m-s}\psi_I \binom {m-1} {m-s-\sum I} \sum_{J\le I} \left(-\frac{1}{2}\right)^{\sum I-\sum J}s_{I,J} s_{\lambda(J)}(x_1,\dots,x_s)&=\\
\sum_{\sum I\le m-s}\psi_I \binom {m-1} {m-s-\sum I}s_{\lambda(I)}(x_1-1/2,\dots,x_s-1/2)&=\\
\sum_{i=\binom{s}{2}}^{m-s}\sum_{\sum I=i} \binom {m-1}{m-s-i} \psi_I s_{\lambda(I)}(x_1-1/2,\dots,x_s-1/2)&=\\
\sum_{i=\binom{s}{2}}^{m-s} \binom {m-1}{m-s-i} \sum_{\sum I=i}  \psi_I s_{\lambda(I)}(x_1-1/2,\dots,x_s-1/2)&=\\
\sum_{i=\binom{s}{2}}^{m-s} \binom {m-1}{m-s-i} s_{(i-\binom{s}{2})}(x_i+x_j-1\vert 1\le i\le j\le s)&=\\s_{(m-s-\binom{s}{2})}(x_i+x_j\vert 1\le i\le j\le s).
\end{align*}

In the last equality we applied Lemma \ref{shiftedcomplete} to the variables $x_i+x_j-1$. %
\end{proof}

Now we are able to present the proof of Theorem \ref{nrs}:

\begin{proof}[Proof of Theorem \ref{nrs}]
We replace $b_I(n)$ by the expression from Lemma \ref{lemmab_I}, change the order of summation and use Lemma \ref{lemmasij} in the last step:

\begin{align*}
&\sum_{\sum I\le m-s} (-1)^{m-s-\sum I} \psi_I b_I(n)\binom {m-1} {m-s-\sum I}=\\ &=\sum_{\sum I\le m-s} \sum_{J\le I} s_{I,J} \psi_{[n]\setminus J} \left(\frac{1}{2}\right)^{\sum I-\sum J} (-1)^{m-s-\sum I} \psi_I \binom {m-1}{m-s-\sum I}
\\& =\sum_{\sum J\le m-s}(-1)^{m-s-\sum J} \psi_{[n]\setminus J}\sum_{\substack{I\ge J\\ \sum I\le m-s}}s_{I,J} \left(-\frac{1}{2}\right)^{\sum I-\sum J}  \psi_I\binom {m-1}{m-s-\sum I}
\\& \hspace*{3cm}=\sum_{\sum J=m-s}(-1)^{m-s-\sum J}\psi_{[n]\setminus J}\psi_{J}=\delta(m,n,n-s).  
\end{align*}
\end{proof}

\section{General square matrices}\label{sec:typeA}
The results from the previous sections have natural analogues if we replace the space of symmetric matrices (``type C'') with the space of skew-symmetric matrices (``type D''), or with the space of general matrices (``type A''). 
This section will be devoted to the latter case, and the next section to the former one.

\subsection{Codegrees of smooth determinantal loci}
Let $M_n$ denote the space of complex matrices of size $n$, and $D^{n-r,n}\subset \mathbb{P}(M_n)$ the locus of matrices 
of rank at most $n-r$. Denote by $D^{n-r,n}_{m}$ its intersection with a general $m$-dimensional projective space. 
Its dimension is $d=m-r^2$ when this is non negative, otherwise it is empty. 
The analogues of the Pataki's inequalities are given by:

\begin{prop} 
The dual variety of $D^{n-r,n}_{m}$ is a hypersurface if and only if 
$$ r^2\le m\le n^2-(n-r)^2.$$
\end{prop}

As was done in \cite{NRS} for symmetric matrices, the degree of this dual variety can be computed by classical means when $D^{n-r,n}_{m}$ is smooth, 
which is equivalent to $r^2\le m\le r^2+2r$. The class formula gives, in terms of topological Euler characteristics, 
$$\deg (D^{n-r,n}_{m})^*= (-1)^d\Big( \chi(D^{n-r,n}_{m})-2 \chi(D^{n-r,n}_{m-1})+ \chi(D^{n-r,n}_{m-2})\Big).$$
Euler characteristics of smooth degeneracy loci have been computed by Pragacz \cite{pragacz1988enumerative}. For $\varphi: F\rightarrow E$ a 
morphism of vector bundles of ranks $f,e$ over a variety $X$, the formula given in  \cite[page 57]{Fulton_Pragacz} is 
$$\chi(D_r(\varphi))= \int_X P_r(E,F)c(X),$$
where $c(X)$ denotes the total Chern class of $X$, while 
$P_r(E,F)$ is a universal polynomial in the Chern classes of $E$ and $F$. Explicitely,
$$P_r(E,F)=\sum_{\lambda,\mu}(-1)^{|\lambda|+|\mu|}D_{\lambda,\mu}^{n-r,m-r}s_{(m-r)^{n-r}+\lambda,\tilde{\mu}}(E-F),$$
where the sum is over partitions $\lambda$ and $\mu$ of length $n-r$ and $m-r$ respectively, and $\tilde{\mu}$ is the 
dual partition of $\mu$. Moreover the coefficients denoted $D_{\lambda,\mu}^{n-r,m-r}$ in \cite{Fulton_Pragacz} encode the 
Segre classes of a tensor product of vector bundles. (We will 
rather use in the sequel the notations of \cite{LLT}, see Definition \ref{LascouxA}.)

We want to apply this formula to $D^{n-r,n}_{m}$, which we consider as the degeneracy locus $D_{n-r}(\varphi)$ 
of the tautological morphism $\varphi : F=\mathcal{O}(-1)^{\oplus n}\longrightarrow \mathcal{O}^{\oplus n}$ over $X=\mathbb{P}^m$. 
Since $c(\mathbb{P}^m)-2hc(\mathbb{P}^{m-1})+h^2c(\mathbb{P}^{m-2})=(1+h)^{m-1}$, 
where $h$ denotes the hyperplane class, we get the 
formula 
$$\deg (D^{n-r,n}_{m})^*=\sum_{\lambda,\mu}(-1)^{|\lambda|+|\mu|}
\binom{m-1}{r^2+|\lambda|+|\mu|}D_{\lambda,\mu}^{r,r}
s_{(r)^r+\lambda,\tilde{\mu}}(\underbrace{1,\ldots ,1}_{n\;\mathrm{times}}),$$
the sum being taken  over partitions $\lambda$ and $\mu$ of length $r$. Note that the dependence on $n$ for $r$ and $m$
fixed is only in the last term, more precisely in the number of one's on which the Schur functions are evaluated. 
This dependence is well known to be polynomial in $n$; very explicitely, for any partition $\nu$, 
$$s_\nu(\underbrace{1,\ldots ,1}_{n\;\mathrm{times}})=\dim S_\nu\mathbb{C}^n = c_\nu(n)/h(\nu),$$
where $c_\nu$ is the content polynomial and $h(\nu)$ is the product of the hook lengths of $\nu$ \cite{Macdonald}. 
A priori this formula is only valid in the range $r^2\le m\le r^2+2r$, when $D^{n-r,n}_{m}$ is smooth.  That it should be true in general would be an analogue of the NRS conjecture in type A. We will prove below that this statement is correct.  

We introduce the following notations, similar to those we used for symmetric matrices. 
\begin{defi}
	We define $\delta_A(m,n,r)$ to be the degree of the variety $(D_{m}^{r,n})^*$ if it is a hypersurface, and zero otherwise. Here $D_{m}^{r,n}$ is the variety of $n\times n$ matrices of rank at most $r$, intersected with a general  
(projective) $m$ dimensional subspace. %
	Equivalently, if we let $Z_r \subset \mathbb{P}(V^* \ot V) \times \mathbb{P}(V^* \ot V)$ be the variety of pairs of matrices $(X,Y)$, up to scalars, with $X\cdot Y= Y \cdot X = 0$, $\rk X \leq r$, $\rk Y \leq n-r$, then the multidegree of $Z_r$ is equal to
	\[[Z_r]=\sum_m{\delta_A(m,n,r)H_1^{n^2-m}H_{n-1}^{m}},\] 
	where $H_1$ and $H_{n-1}$ denote the pull-backs of the hyperplane classes from 
	$\mathbb{P}(V^* \ot V)$ and $\mathbb{P}(V^* \ot V)$, respectively.
\end{defi}
\begin{defi}
The number $\phi_A(n,d)$ is the degree of the variety $\fL^{-1}$, where $\fL \subseteq \mathbb{P}(M_n)$ is a general linear subspace of dimension $d-1$.
\end{defi}

\subsection{Complete collineations}
The correct space to work with is the \emph{space of complete collineations} \cite{sempleCompleteCollineations,Tyrrell,VainsencherCompleteCollineations,LaksovCompleteQuadrics,Thaddeus2,Massarenti1}. It can actually be defined for rectangular matrices, but for sake of simplicity we will restrict ourselves to square matrices.

\begin{defi}
Let $V$ and $W$ be two vector spaces of equal dimension $n$. The space $\mathbb{P}(V^* \ot W)$ represents linear maps from $V$ to $W$; the open subset of rank $n$ linear maps is denoted by $\mathbb{P}(V^* \ot W)^\circ$. Then the space of complete collineations $\ComplCol(V,W)$ is defined as the closure of the image of the map
$$
\phi : \mathbb{P}(V^* \ot W)^\circ  \to \mathbb{P}(V^* \ot W) \times \mathbb{P}\left(\bigwedge^2{V^*} \ot \bigwedge^2{W}\right) \times\ldots\times \mathbb{P}\left(\bigwedge^{n-1}{V^*} \ot \bigwedge^{n-1}{W}\right),
$$
given by
$$
[A]\mapsto ( [A],[\wedge^2 A],\ldots, [\wedge^{n-1}A]).
$$
As before, in coordinates this map sends a matrix to its minors of various sizes.
\end{defi}

As in the symmetric case, the space of complete collineations can be constructed by blowing-up $\PP(M_n)$ 
along the subvariety of rank one matrices, then the strict transform of the subvariety of matrices of
rank at most two, and so on. As such, it admits a first series $S_1,\ldots,S_{n-1}$ of special classes of divisors: the classes of (the strict transforms of) the exceptional divisors $E_1,\ldots,E_{n-1}$
of these successive blow-ups. 
A second natural series $L_1,\ldots,L_{n-1}$ of classes of
divisors can be obtained by pulling back the hyperplane classes under the projections $\pi_i:\ComplCol(V,W) \to \mathbb{P}\left(\bigwedge^{i}{V^*} \ot \bigwedge^{i}{W}\right)$.

The analogue of \Cref{prop:relations} holds:
\begin{prop}\label{prop:relationsA}
$L_1,\ldots, L_{n-1}$ form a basis of $\Pic(\ComplCol(V,W))$, in which the $S_i$'s are given by the formulas
	$$
	S_i= -L_{i-1}+2L_i-L_{i+1},
	$$
	with the convention that $L_0=L_n:=0$.
\end{prop}
\begin{proof}
 Follows from \cite[Proposition 3.6, Theorem 3.13]{Massarenti1}.
\end{proof}

\begin{prop} \label{prop:PhiDeltaChowA} The numbers $\phi_A$ and $\delta_A$ can be computed as intersection products 
of the variety of complete collineations:
	$$
	\phi_A(n,d) = \int_{CC_n}L_1^{n^2-d} L_{n-1}^{d-1},
	$$
	$$
	\delta_A(m,n,r)=\int_{CC_n}S_{r}L_1^{n^2-m-1}L_{n-1}^{m-1}=\int_{E_r}L_1^{n^2-m-1}L_{n-1}^{m-1}.
	$$
\end{prop}
This implies the analogue of \cref{eq:phitodeltaAlt}:
\begin{equation}\label{eq:phitodeltaA}
\phi_A(n,d) = \frac{1}{n}\sum_{r=1}^{n-1}{r\delta_A(d,n,n-r)}.
\end{equation}
\begin{defi}\label{LascouxA}
	We define type $A$ Lascoux coefficients $d_{I,J}$ as follows. For $X=(x_1,\ldots,x_k)$ and $Y=(y_1,\dots,y_l)$ two sets of indeterminates, we denote by $X+Y$ the set of indeterminates $x_i+y_j$, $1\le i\le k$, $1\le j\le l$. Then the $d_{I,J}$'s are defined by the formal identity 
	\[
	s_{(d)}(X+Y)=\sum_{\substack{\# I=k, \# J=l\\ |\lambda(I)|+|\lambda(J)|=d}} 
	d_{I,J} s_{\lambda(I)}(X)s_{\lambda(J)}(Y).
	\]
	Equivalently, for the product of  the universal bundles $\fU_1\otimes \fU_2$ over a product of 
	Grassmannians $G(k,m)\times G(l,n) $:
	\[
	Seg_d(\fU_1\otimes \fU_2) = \sum_{\substack{\# I=k, \# J=l\\ |\lambda(I)|+|\lambda(J)|=d}}{d_{I,J}\sigma^1_{\lambda(I)}\sigma^2_{\lambda(J)}}
	\]
	\end{defi}
Analogously to \Cref{thm:delta}, we have the following formula for $\delta_A$: 
\begin{thm}\label{thm:formuladeltaA}
	\[
	\delta_A(m,n,r) = \sum_{\substack{I,J \subset [n] \\ \# I=\# J=n-r \\ \sum{I}+\sum{J}=m-n+r}}{d_{I,J} d_{[n] \setminus I, [n] \setminus J}}
	\]
\end{thm}

\subsection{Induction relations and polynomiality}

We denote by $D(t)$ the infinite matrix with entries $D(t)_{ij}={{t+i+j}\choose i}$. This matrix gives us a formula for $d_{I,J}$ \cite[Proposition 2.8]{LLT}. %
\begin{prop} \label{prop:d_pascalformula} 
Let $I=\{i_1,\dots,i_r\},\ J=\{j_1,\dots,j_s\}$ be two sets of nonnegative integers with $r\le s$. Then
\[ d_{I,J}=\begin{cases} \det D(s-r)_{I,\{j_{s-r+1}-(s-r),\dots,j_s-(s-r)\}} &\text{ if } j_i=i-1 \text{ for all } 1\le i\le s-r\\ 0 &\text{ otherwise .}
\end{cases}\]

In particular, if $\# I=\# J$ then $d_{I,J}=\det D(0)_{I,J}$.
\end{prop}
\begin{lemma} \label{lem:typeArecursion}
\begin{enumerate}
\item Let $I=\{i_1,\dots,i_s\},\ J=\{j_1,\dots,j_s\}$ with $i_1,j_1 > 1$. Write $I_0=\{0\} \cup I$ and $J_0=\{0\} \cup J$. Then 
\[
d_{I,J}=(s+1)d_{I_0,J_0}-\sum_{p=1}^{s}{d_{I_0 \setminus \{i_p \} \cup \{i_p-1\},J_0}}-\sum_{q=1}^{s}{d_{I_0,J_0 \setminus \{j_q \} \cup \{j_q-1\}}}.
\]
(Here, if $I_0 \setminus \{i_p \} \cup \{i_p-1\}$ is a multiset, then $d_{I_0 \setminus \{i_p \} \cup \{i_p-1\},J_0}=0$.)
\item For $i_1=0$ or $j_1=0$ we have: 
\[
d_{\{i_1,\ldots,i_s\},\{j_1,j_2,\ldots,j_s\}}=\sum_{\substack{i_\ell \leq i'_\ell < i_{\ell+1}\\j_\ell \leq j'_\ell < j_{\ell+1}}} {d_{\{i'_1,\ldots,i'_{s-1}\},\{j'_1,\ldots,j'_{s-1}\}}}.
\]
\end{enumerate}
\end{lemma} 
\begin{proof}
	\begin{enumerate}
		\item We expand the determinant $\det D(0)_{I_0,J_0}$ in each row, and sum up:
		\begin{align*}
		(s+1) d_{I_0,J_0} =& \sum_{p,q=0}^{s}{(-1)^{p+q}\binom{i_p+j_q}{i_p}d_{I_0\setminus\{i_p\},J_0\setminus\{j_q\}}} \\
		= & d_{I,J}+\sum_{p=1}^{s}{(-1)^{p}d_{I_0\setminus\{i_p\},J}}+\sum_{q=1}^{s}{(-1)^{q}d_{I,J_0\setminus\{j_q\}}}\\ & +\sum_{p,q=1}^{s}{(-1)^{p+q}\left(\binom{i_p+j_q-1}{i_p}+\binom{i_p+j_q-1}{i_p-1}\right)d_{I_0\setminus\{i_p\},J_0\setminus\{j_q\}}} \\
		= & d_{I,J}+\sum_{p=1}^{s}\sum_{q=0}^{s}{(-1)^{p+q}\binom{i_p+j_q-1}{i_p-1}d_{I_0\setminus\{i_p\},J_0\setminus\{j_q\}}} \\ & +\sum_{q=1}^{s}\sum_{p=0}^{s}{(-1)^{p+q}\binom{i_p+j_q-1}{i_p}d_{I_0\setminus\{i_p\},J_0\setminus\{j_q\}}} \\ = & 
		d_{I,J}+\sum_{p=1}^{s}{d_{I_0 \setminus \{i_p \} \cup \{i_p-1\},J_0}} + \sum_{q=1}^{s}{d_{I_0,J_0 \setminus \{j_q \} \cup \{j_q-1\}}}.
		\end{align*}
\item The proof of the second formula is similar to the proof of formula \ref{eq:rec2} in Lemma \ref{lem:recurssionsPSI}. We only consider the case $i_1=0$ and in $s_{(d)}(\{x_i + y_j \mid 1 \le i,j \le s\})$ we subsitute $x_s=0$. This yields
\[d_{\{i_1,\ldots,i_s\},\{j_1,j_2,\ldots,j_s\}}=\sum_{j_{\ell-1} < j'_\ell \le j_{\ell}} {d_{\{i_2-1,\ldots,i_s-1\},\{j'_1,\ldots,j'_{s}\}}}.\]

Then by Proposition \ref{prop:d_pascalformula} all summands with $j'_1>0$ are zero. This allows to substitute $y_s=0$ in $s_{(d)}(\{x_i + y_j \mid 1 \le i\le s-1,1\le j\le s\})$ and conclude the lemma analogously to  formula \ref{eq:rec2} in Lemma \ref{lem:recurssionsPSI}.
	\end{enumerate}
\end{proof}
\begin{thm} \label{thm:polypsiA}
Let $I = \{i_1,\ldots,i_r\}, J=\{j_1,\dots,j_r\}$ be two sets of strictly increasing nonnegative integers.
The function defined for $n\geq 0$ by
\[
LP^A_{I,J}(n):=
\begin{cases}
	d_{[n] \setminus I,[n] \setminus J} &\text{ if } I,J\subset [n], \\
	0 &\text{ otherwise},
	\end{cases}
	\]
 is polynomial in $n$. %
 \end{thm}
 \begin{proof}
 	From \Cref{lem:typeArecursion} it follows that
 	\begin{multline*}
 	LP^A_{I,J}(n)=(n-r+1)LP^A_{I \setminus \{0\},J \setminus \{0\}}(n) \\ -\sum_{\ell: i_{\ell+1}>i_\ell+1}{LP^A_{I \setminus \{0,i_\ell\} \sqcup \{i_\ell+1\},J\setminus\{0\}}(n)} - \sum_{\ell: j_{\ell+1}>j_\ell+1}{LP^A_{I\setminus\{0\},J \setminus \{0,j_\ell\} \sqcup \{j_\ell+1\}}(n)}
 	\end{multline*}
 	if $i_0=j_0=0$, and otherwise
 	\[
 	LP^A_{I,J}(n)=\sum_{I',J'}{LP^A_{I',J'}(n-1)},
 	\]
 	where the sum is over all pairs $(I',J')$ of the form $(\{i_1-\epsilon_1, \ldots, i_r-\epsilon_r\},\{j_1-\mu_1, \ldots, j_r-\mu_r\})$ with $\epsilon_\ell,\mu_\ell \in \{0,1\}$. As in the first proof of \Cref{thm:polyPSI}, it follows by induction that $LP^A_{I,J}$ is polynomial. 
 \end{proof}

\begin{thm}\label{thm:polyDeltaA}
	For every fixed $m,s$, the function $\delta_A(m,n,n-s)$ is a polynomial in $n$. 
\end{thm}
\begin{proof}
	Follows from Theorems \ref{thm:formuladeltaA} and \ref{thm:polypsiA}.
\end{proof}
\begin{thm}\label{thm:polyPHIA}
	For any fixed $d$, the function $\phi_A(n,d)$ is a polynomial for $n>0$.  
\end{thm}
\begin{proof}
Follows from \cref{eq:phitodeltaA} and \Cref{thm:polyDeltaA}.
\end{proof}

\subsection{Proof of the NRS  conjecture for type A} 

We start with the following analogue of Lemma \ref{lemmasij}. 

\begin{lemma} \label{lemmadij}
Let $K, L$ be sets of $r$ nonnegative integers,  with $\sum K+ \sum L\le m-r$. Then 
\[\sum_{I\ge K}d_{I,L} (-1)^{\sum I-\sum K}s_{I,K} 
\binom {m-1} {m-r-\sum I-\sum L}=d_{K,L}\]
if $\sum K+ \sum L=m-r$, while this sum vanishes if $\sum K+ \sum L<m-r.$
\end{lemma}

\begin{proof}
Using Lemma \ref{shiftedcomplete} we compute $s_{(m-r^2)}(X+Y)$ as 
\begin{align*}
 \sum_{k=r^2}^{m} \binom {m-1}{m-k} s_{(k-r^2)}(X+Y-1)=\hspace*{6cm}  \\
= \sum_{k=r^2}^{m} \binom {m-1}{m-k} \sum_{\sum I+\sum L=k-r}  d_{I,L} s_{\lambda(I)}(X-1)
s_{\lambda(L)}(Y) =\hspace*{1cm} \\ 
 \hspace*{1cm}=\sum_{k=r^2}^{m} \binom {m-1}{m-k} \sum_{\sum I+\sum L=k-r} 
 d_{I,L}  \sum_{K\le I}(-1)^{\sum I-\sum K} s_{I,K}s_{\lambda(K)}(X)s_{\lambda(L)}(Y).
\end{align*}
Comparing this expansion with that of Definition \ref{LascouxA} yields the claim. 
\end{proof}

\begin{defi}
For two partitions $\lambda=\lambda(I)$, $\mu=\lambda(J)$, of length $r$, we define the polynomial
$$a_{I,J}(n):=s_{(r)^r+\lambda,\tilde{\mu}}(\underbrace{1,\ldots ,1}_{n\;\rm times}).$$
\end{defi}

\begin{lemma}\label{lemmaa_IJ}
Let $I,J$ be two sets of $r$ nonnegative integers. Then 
\[a_{I,J}(n)=\sum_{L\le I}s_{I,L}d_{[n]\setminus L, [n]\setminus J},\]
\[d_{[n]\setminus I, [n]\setminus J}=\sum_{L\le I}(-1)^{\sum I-\sum L}s_{I,L}a_{L,J}(n).\]
\end{lemma}

\begin{proof}
These two formulas being equivalent, we will prove the first one. 
We use the projection formula given in \cite[Lemma 3.1]{pragacz1988enumerative}, \cite[Proposition 1, page 51]{Fulton_Pragacz}, where one considers two vector bundles $E, F$
 of respective  ranks $n,m$ over some variety $X$. The Grassmann bundles $G_sE$ and $G^sF$ (parametrizing rank $s$ subspaces and rank $s$ quotients respectively) over $X$ admit tautological and quotient bundles $S_E,
 Q_E$ and $S_F, Q_F$, where $S_E$ and $Q_F$ have rank $s$. Let $\tau : G:= G_sE\times_X G^sF\rightarrow X$ denote the total projection. The variety $G$ is 
 endowed with the vector bundle $H:=\tau^*Hom(F,E)/Hom(Q_F,S_E)$. 
 In our situation we will suppose that $m=n$, and let $r=n-s$,
 which is the rank of both $S_F$ and $Q_E$.
 Then the projection formula asserts that  
 for any two partitions $\lambda, \mu$ of length at most $r$,
 $$s_{(r)^r+\lambda,\tilde{\mu}}(E-F)=\tau_*\Big(s_\lambda(Q_E)
 s_{\mu}(S_F)c_{top}(H)\Big).$$ 
 (In the original formula $s_{\mu}(S_F)$ is replaced by  $s_{\tilde{\mu}}(-S_F)$, but they are equal.) We can replace $c_{top}(H)$ by $c(H)$ and keep only the term of the correct degree, which yields
 $$%
 c(Hom(F,E))\tau_*\Big(
 s_\lambda(Q_E)s_{\mu}(S_F)s(Hom(Q_F,S_E)\Big)_{|\mathrm{deg}=r^2+|\lambda|+|\mu|}.$$
 Now we specialize to the case where $E=E_0\otimes L$ and $F=F_0$,
 where $E_0, F_0$ are trivial vector bundles of rank $n$, and $L$ is 
 a line bundle on $X$ with $c_1(L)=\delta$. In this case 
 $$c(Hom(F,E))=(1+\delta)^{s^2}.$$
 The Grassmann bundles $G_sE$ and $G^sF$ are then the trivial bundles $G_sE_0\times X$ and $G^sF=G^sF_0\times X$ respectively, while the tautological and quotient bundles are $S_E=S_{E_0}\otimes L$, $Q_E=Q_{E_0}\otimes L$,
 $S_F=S_{F_0}$ and $Q_F=Q_{F_0}$, where we omit the obvious pullbacks.
 In this situation, 
 $$s(Hom(Q_F,S_E))=\sum_{\ell\ge 0}(1+\delta)^{-s^2-\ell}
 s_\ell(Hom(Q_{F_0},S_{E_0})).$$
Let $\lambda =\lambda(I)$ and  $\mu =\lambda(J)$. Using Definition \ref{LascouxA}, Definition \ref{def:sIJ} and the duality properties of Schubert classes, we deduce that 
$$s_{(r)^r+\lambda,\tilde{\mu}}(E-F)=a_{I,J}(n)\delta^{r^2+|\lambda|+|\mu|}$$ can be computed by picking 
the term of the correct degree in 
$$\sum_{\ell\ge 0}(1+\delta)^{n^2-s^2-\ell}\sum_L\delta^{\sum I-\sum L}s_{I,L}d_{[n]/J,[n]/L},$$
where the size of $L$ is constrained by the relation $\ell + \sum L=
n(n-1)-s(s-1)-\sum J$. This implies that $n^2-s^2-\ell+\sum I-\sum L=
r+\sum I+\sum J=r^2+|\lambda|+|\mu|$. So the term of the correct degree is actually the term of maximal degree in $\delta$, and the claim follows. 
\end{proof}

\medskip We are now ready to prove the NRS Conjecture in type A.

\begin{thm} \label{thm:NRSA}
$$\delta_A(m,n,n-r)=\sum_{\substack{\# I=\# L=r,\\ \sum I+\sum L\leq m-r}}
d_{I,L}(-1)^{m-r-\sum I-\sum L}\binom {m-1} {m-r-\sum I-\sum L}a_{I,L}(n).$$
\end{thm}

\begin{proof}
Using Theorem \ref{thm:formuladeltaA}, Lemma \ref{lemmadij} and Lemma \ref{lemmaa_IJ}, we get 
\begin{align*}
 & \delta_A(m,n,n-r) =  \sum_{\substack{K,L\\ \sum K+\sum L=m-r}}d_{K,L}d_{[n]/K,[n]/L} =& \\
 & \hspace*{2cm}= \sum_{\substack{K,L\\ \sum K+\sum L\leq m-r}}(-1)^{m-r-\sum K-\sum L}\sum_{I\ge K}\sum_L d_{I,L}\times \\
  & \hspace*{4cm}\times (-1)^{\sum I-\sum K}s_{I,K} 
\binom {m-1}{m-r-\sum I-\sum L}d_{[n]/K,[n]/L} \\
& \hspace*{2cm}=\sum_{\substack{I,L\\ \sum I+\sum L\leq m-r}}(-1)^{m-r-\sum I-\sum L}d_{I,L}\times \\
& \hspace*{4cm}\times\binom {m-1}{m-r-\sum I-\sum L}\sum_{K\le I}  (-1)^{\sum I-\sum K}s_{I,K} 
d_{[n]/K,[n]/L} \\
& \hspace*{2cm}=\sum_{\substack{I,L\\ \sum I+\sum L\leq m-r}}(-1)^{m-r-\sum I-\sum L}d_{I,L}
\binom {m-1}{m-r-\sum I-\sum L}a_{I,L}(n).
\end{align*}
\end{proof}

\section{Skew-symmetric matrices}

\subsection{Codegrees of smooth skew-symmetric determinantal loci}
Let $A_n$ denote the space of skew-symmetric complex matrices of size $n$, and $AD^{n-r,n}\subset \mathbb{P}(A_n)$ the locus of matrices 
of rank at most $n-r$, where $n-r$ is always supposed to be even. Denote by $AD^{n-r,n}_{m}$ its intersection with a general $m$-dimensional projective space. 
Its dimension is $d=m-\binom{r}{2}$ when this is non negative, otherwise it is empty. 
The analogs of the Pataki's inequalities are given by:

\begin{prop} 
The dual variety of $AD^{n-r,n}_{m}$ is a hypersurface if 
and only if 
$$ \binom{r}{2}\le m\le \binom{n}{2}-\binom{n-r}{2}.$$
\end{prop}

As in the previous case, the degree of this dual variety can be computed by classical means when $AD^{n-r,n}_{m}$ is smooth, 
which is equivalent to $\binom{r}{2}\le m\le \binom{r}{2}+2r$. The class formula gives, in terms of topological Euler characteristics, 
$$\deg (AD^{n-r,n}_{m})^*= (-1)^d\Big( \chi(AD^{n-r,n}_{m})-2 \chi(AD^{n-r,n}_{m-1})+ \chi(AD^{n-r,n}_{m-2}\Big).$$
Euler characteristics of smooth skew-symmetric degeneracy loci have also been computed by Pragacz \cite{pragacz1988enumerative}. For $E$ a vector bundle of rank $e$
over a variety $X$, and $\varphi: E^*\rightarrow E$ a 
skew-symmetric morphism, the formula given in  \cite[page 64]{Fulton_Pragacz} is 
$$\chi(D_s(\varphi))= \int_X P_s(E)c(X),$$
where $c(X)$ denotes the total Chern class of $X$, 
while $P_s(E)$ is a universal polynomial in the Chern classes of $E$. Explicitely,
$$P_s(E)=\sum_{\ell(\lambda)\le n-s}(-1)^{|\lambda|}
[\lambda+\rho(n-s-1)]P_{\lambda+\rho(n-s-1)}(E),$$
where the coefficients $[\lambda+\rho(n-s-1)]$ are those appearing
in the Segre class of the skew-symmetric square of a vector bundle of rank $n-s$. These coefficients were denoted $\alpha_I$ in \cite{LLT}, that we will rather follow, where $I$ is a set of $r=n-s$ nonnegative integers.

Let us apply this formula to $AD^{n-r,n}_{m}$, which we consider formally as the degeneracy locus $D_{s}(\varphi)$ 
of the tautological skew-symmetric morphism $\phi : F=\mathcal{O}(-\frac{1}{2})^{\oplus n}\longrightarrow \mathcal{O}(\frac{1}{2})^{\oplus n}$ over $X=\mathbb{P}^m$. 
Since $c(\mathbb{P}^m)-2hc(\mathbb{P}^{m-1})+h^2c(\mathbb{P}^{m-2})=(1+h)^{m-1}$, 
where $h$ denotes the hyperplane class, we get the 
formula 
$$\deg (AD^{n-r,n}_{m})^*=\sum_{I}\binom{m-1}{m-\sum I}\alpha_I 
P_I(\underbrace{1,\ldots ,1}_{n\;\mathrm{times}}),$$
where the sum goes over the sets $I$ of $r$ nonnegative integers.
Once again the dependence on $n$ for $r$ 
fixed is only in the last term, more precisely in the number of one's on which the P-Schur functions are evaluated. 
We have already seen that this dependence is well known to be polynomial in $n$. 

A priori this formula is only valid in the range $\binom{r}{2}\le m\le \binom{r}{2}+2r$, when $AD^{n-r,n}_{m}$ is smooth.  That it should be true in general would be an analogue of the NRS conjecture in type D. We will prove below that this statement is correct. Our notations for the dual degrees will be as follows

\begin{defi}
	Define $\delta_D(m,n,r)$ to be the degree of the variety $(AD_{m}^{2r,2n})^*$ if it is a hypersurface, and zero otherwise. Here $AD_{m}^{2r,2n}$ is the variety of rank at most $2r$ skew-symmetric $2n\times 2n$ matrices, intersected with a general  
(projective) $m$ dimensional subspace. %
	Equivalently, if we let $Z_{r} \subset \mathbb{P}(\wedge^2 V^*) \times \mathbb{P}(\wedge^2 V)$ be the variety of pairs of matrices $(X,Y)$, up to scalars, with $X\cdot Y=0$, $\rk X \leq 2r$, $\rk Y \leq n-2r$. Then the multidegree of $Z_r$ is equal to
	\[
	[Z_r]=\sum_m{\delta_D(m,n,r)H_1^{\binom{n}{2}-m}H_{n-1}^{m}},
	\]
	where $H_1$ and $H_{n-1}$ are the pullbacks of the hyperplane classes from 
	 $\mathbb{P}(\wedge^2 V^*)$ and $\mathbb{P}(\wedge^2 V)$.
\end{defi}

\subsection{Complete skew-symmetric forms}
A well-known particularity of skew-symmetric forms is that the cases of odd and even sizes are
quite different. In particular the following definition only makes sense in the even case.

\begin{defi}
	The number $\phi_D(n,d)$ is the degree of the variety $\fL^{-1}$, where $\fL \subseteq \mathbb{P}(A_{2n})$ is a general linear subspace of dimension $d-1$.
\end{defi}

In this section, we will be only working with skew-symmetric matrices of even size $2n\times 2n$.
The relevant space to deal with is then the \emph{space of complete skew-forms}. Just as with complete quadrics, there are many ways of constructing this space. Here we give just two, referring the reader to the literature \cite{Bertram,Thaddeus2,Massarenti2} for other equivalent definitions.
\begin{defi}
	Let $V$ be a $2n$-dimensional vector space. 
	The space of complete skew-forms $\CS(V)$ is defined as the closure of 
	$\phi(\mathbb{P}(\bigwedge^2(V))^\circ)$, where
	$$
	\phi : \mathbb{P}\left(\bigwedge^2V\right)^\circ  \to \mathbb{P}\left(\bigwedge^2V\right) \times \mathbb{P}\left( \bigwedge^4 V \right) \times\ldots\times \mathbb{P} \left(\bigwedge^{2n-2} V \right),
	$$
	given by
	$$
	[A]\mapsto ([A],[\wedge^2 A],\ldots, [\wedge^{n-1}A]).
	$$
	
	We note that here $\wedge^i A$ is viewed as an element of $\bigwedge^{2i}V$, %
	 see also \cite[Section 3]{Bertram}. In coordinates, the map $\bigwedge^2V \to \bigwedge^{2i}V$ sends the entries of a skew-symmetric matrix to the Pfaffians of its principal $2i \times 2i$ submatrices.

For simplicity we will also use the notation $CS_{2n}=CS(\mathbb{C}^{2n})$. 
\end{defi}

As in the smmetric case, the space of complete skew-forms can be constructed by blowing up $\PP(A_{2n})$ 
along the subvariety of rank two matrices, then the strict transform of the subvariety of matrices of
rank at most four, and so on. As such, it admits a first series $S_1,\ldots,S_{n-1}$ of special classes of divisors: the classes of (the strict transforms of) the exceptional divisors $E_1,\ldots,E_{n-1}$
of these successive blow-ups. 
A second natural series $L_1,\ldots,L_{n-1}$ of classes of
divisors can be obtained by pulling back the hyperplane classes under the projections $\pi_i:\CS(V) \to \mathbb{P}\left(\bigwedge^{2i}V\right)$.

The analogue of \Cref{prop:relations} holds:
\begin{prop}\label{prop:relationsD}
The classes $L_1,\ldots, L_{n-1}$ for a basis $\Pic(\CS(V))$, in which the $S_i$'s are given by 
	$$
	S_i= -L_{i-1}+2L_i-L_{i+1},
	$$
	with $L_0=L_n:=0$.
\end{prop}
\begin{proof}
 Follows from \cite[Proposition 3.6, Theorem 3.9]{Massarenti2}.
\end{proof}

As with symmetric matrices, the numbers $\phi_D$ and $\delta_D$ can be expressed as intersection products in the Chow ring of $CS_{2n}$:
\begin{prop} \label{prop:PhiDeltaChowD}
	$$
	\phi_D(n,d) = \int_{CS_{2n}}L_1^{\binom{2n}{2}-d} L_{n-1}^{d-1}
	$$
	$$
	\delta_D(m,n,r)
	=\int_{CS_{2n}}S_{r}L_1^{\binom{2n}{2}-m-1}L_{n-1}^{m-1}
	=\int_{E_r}L_1^{\binom{2n}{2}-m-1}L_{n-1}^{m-1}.
	$$
\end{prop}
\begin{proof}
	Analogous to the proof of \Cref{prop:PhiDeltaChow}.
\end{proof}
From the two propositions above and the Pataki inequalities, we deduce that 
\begin{equation}\label{eq:phitodeltaD}
\phi_D(n,d) = \frac{1}{n}\sum_{\binom{r}{2}\le d}{r\delta_D(d,n,n-r)},
\end{equation}
the analogue of \cref{eq:phitodeltaAlt}.

\begin{defi}\label{LascouxD}
	We define type $D$ Lascoux coefficients $\alpha_{I}$ as follows. For $X=(x_1,\ldots,x_k)$ a set of indeterminates, we denote by $\lambda(X)$ the set of indeterminates $x_i+x_j$, $1\le i<j\le  k$.
	 Then the $\alpha_{I}$'s are defined by the formal identity 
	\[
	s_{(d)}(\lambda(X))=\sum_{\substack{\# I=k, \\ |\lambda(I)|=d}} 
	\alpha_{I} s_{\lambda(I)}(X).
	\]
	Equivalently, for  the universal bundle $\fU$ over a
	Grassmannian $G(k,m)$,
	\[
	Seg_d\left(\bigwedge^2\fU\right) = \sum_{\substack{\# I=k\\ |\lambda(I)|=d}}  \alpha_{I} \sigma_{\lambda(I)}.
	\]
	\end{defi}

For more about these coefficients, see \cite[Proposition A.16]{LLT}.

\begin{thm}\label{thm:formuladeltaD}
	\[
	\delta_D(m,n,r) = \sum_{\substack{I \subset [2n] \\ \#I=2n-2r \\ \sum{I}=m}}{\alpha_I \alpha_{[2n] \setminus I}}
	\]
\end{thm}
\begin{proof}
	Analogous to the proof of \Cref{thm:delta}.
\end{proof}

\subsection{Induction relations and polynomiality}
We will now prove the polynomiality (or more precisely, quasipolynomiality) of $\alpha_{[k] \setminus I}$. The following recursive relations will be central to our proof:
\begin{lemma}\label{lem:recurssionsPSID}
\begin{enumerate}
\item For $j_1>0$ we have: 
\begin{equation}
\alpha_{\{j_1,\dots,j_s\}}=\begin{cases}\alpha_{\{0,j_1,\dots,j_s\}} &\text{ if } s \text{ is even }\\
0 &\text{ if } s \text{ is odd }
\end{cases}
\end{equation}
\item For $j_1=0$ we have: \begin{equation}
\alpha_{\{j_1,j_2,\ldots,j_s\}}=\sum_{j_\ell \leq j'_\ell < j_{\ell+1}}{\alpha_{\{j'_1,\ldots,j'_{s-1}\}}}.
\end{equation}
\end{enumerate}
\end{lemma}
\begin{proof}
First formula is \cite[p.~446]{pragacz1988enumerative}, \cite[(A.16.3)]{LLT} and \cite[p.~163-166]{MR1481485}. The proof of the second formula is analogous to the proof of \cref{eq:rec2} in Lemma \ref{lem:recurssionsPSI}. 
\end{proof}

\begin{thm}\label{thm:polypsiD}
Let $I = \{i_1,\ldots,i_s\}$ be a set of strictly increasing nonnegative integers.
For $k\geq 0$ the function:
\[
LP^D_I(k):=
\begin{cases}
	\alpha_{[k] \setminus I} &\text{ if } I\subset [k], \\
	0 &\text{ otherwise}.
	\end{cases}
	\]
 is a quasi-polynomial in $k$ with period 2, i.e. for both even $k$ and odd $k$ it is a polynomial.
 \end{thm}
\begin{proof}
We proceed as in the first proof of Theorem \ref{thm:polyPSI} by induction on $\# I$ and then on $\sum I$ using relations from Lemma \ref{lem:recurssionsPSID}. The difference is that in the case $i_0=0$ we have 
\[LP^D_I(n)=\begin{cases} LP^D_{I\setminus{0}}(n) &\text{ if } n-\# I \text{ is even }\\
0 &\text{ if } n-\# I \text{ is odd } 
\end{cases}\]
which is clearly by induction hypothesis a quasipolynomial in $n$ with period 2. The rest is analogous as in the proof of Theorem \ref{thm:polyPSI}. 
\end{proof}
From \Cref{thm:formuladeltaD} and \Cref{thm:polypsiD} we deduce the polynomiality of $\delta_D$:
\begin{thm}
	For every fixed $m,s$, the function $\delta_D(m,n,n-s)$ is a polynomial in $n$.
\end{thm}
Using \cref{eq:phitodeltaD},we also get the polynomiality of $\phi_D$:
\begin{thm}\label{thm:polyPHID}
	For any fixed $d$, the function $\phi_D(n,d)$ is a polynomial for $n>0$.  
\end{thm}

\subsection{Proof of the NRS conjecture in type $D$} 
The proof of Theorem \ref{nrsD} will be extremely similar 
to that of the original NRS Conjecture.

\begin{lemma} \label{lemmasijTypeD}
Let $J$ be a set of $r$ nonnegative integers,  with $\sum J\le m$. Then 
\[\sum_{\substack{I\ge J\\ \sum I\le m}} \alpha_I \left(-\frac{1}{2}\right)^{\sum I-\sum J}s_{I,J}\binom {m-1} {m-\sum I}=
\begin{cases} 0 &\text{ if } \sum J<m\\
\alpha_J &\text{ if } \sum J=m

\end{cases}\]
\end{lemma}

\begin{proof} Given a set of $r$ variables $X=(x_1,\dots , x_r)$, 
we denote by $\lambda(X)$ the set of variables $(x_i+x_j, i<j)$. 
Using Lemma \ref{shiftedcomplete} we compute $s_{(m-\binom{r}{2})}(\lambda(X))$ as 
\begin{align*}
& \sum_{k=\binom{r}{2}}^{m} \binom {m-1}{m-k} s_{(k-\binom{r}{2})}(\lambda(X)-1)=
 \sum_{k=\binom{r}{2}}^{m} \binom {m-1}{m-k} s_{(k-\binom{r}{2})}(\lambda(X-\frac{1}{2}))=  \\
&\hspace*{1cm}= \sum_{k=\binom{r}{2}}^{m} \binom {m-1}{m-k} \sum_{\sum I=k}   \alpha_Is_{\lambda(I)}(X-\frac{1}{2}) \\
&\hspace*{2cm}= \sum_{k=\binom{r}{2}}^{m} \binom {m-1}{m-k} \sum_{\sum I=k}   \alpha_I\sum_{J\le I}s_{I,J}(-\frac{1}{2})^{\sum I-\sum J}s_{\lambda(J)}(X).
\end{align*}
Comparing this expression the expansion in Definition \ref{LascouxA} yields the claim. 
\end{proof}

\begin{defi}\label{dI}
For  $I$ a set of nonnegative integers, we define $d_I(n)$ by the formula:
$$d_I(n) := P_{I}(\underbrace{1/2,\dots, 1/2}_{n\;\rm times}).$$
\end{defi}

Like the Schur $P$-polynomial themselves \cite[Section III.8]{Macdonald}, these polynomials may be computed recursively.
The following lemma describes the relation between $d_I(n)$ and the Lascoux coefficients $\alpha_J$.

\begin{lemma}\label{lemmad_I}
Let $I$ be a set of $r$ nonnegative integers. Then 
\[d_I(n)=\sum_{J\le I} \left(\frac{1}{2}\right)^{\sum I-\sum J} s_{I,J} \alpha_{[n]\setminus J},\]
\[\alpha_{[n]\setminus I}=\sum_{J\le I} \left(-\frac{1}{2}\right)^{\sum I-\sum J} s_{I,J} d_J(n).\]
\end{lemma}

\begin{proof}
The two formulas are equivalent; we shall prove the first one. 
For a vector bundle $\mathcal{E}$ of rank $n$ over some base $X$, we consider the relative Grassmannian
$G^r(\mathcal{E})$ of rank $r$ quotients of $\mathcal{E}$, with its projection $\pi$ to $X$. 
We denote by $\mathcal{S}$ and $\mathcal{Q}$ the relative tautological subbundle and quotient 
bundle of $\pi^*\mathcal{E}$, of respective ranks $s=n-r$ and $r$. Then for $I=(i_1<\cdots <i_r)$, 
\begin{equation*}\label{projectionD}
P_{I}(\mathcal{E})=\pi_*(c_{top}(\mathcal{S}\otimes\mathcal{Q})P_{I}(\mathcal{Q})),
\end{equation*}
 (cf.~\cite[Example 2, p.~50]{Fulton_Pragacz}). 
Moreover,  \cite[(4.5)]{Fulton_Pragacz}, \cite[Proposition 2.2]{pragacz1991algebro} can be written as  
$$P_{I}(\mathcal{Q})=c_{top}(\wedge^2\mathcal{Q})s_{\lambda(I)}(\mathcal{Q}).$$
Since $\pi^*\mathcal{E}$ is an extension of $\mathcal{Q}$ by $\mathcal{S}$, the bundle $\pi^*(\wedge^2\mathcal{E})$
admits a filtration whose successive quotients are $\wedge^2\mathcal{Q}$, $\mathcal{S}\otimes\mathcal{Q}$ and 
$\wedge^2\mathcal{S}$. Hence the identity 
$$c(\mathcal{S}\otimes\mathcal{Q})c(\wedge^2\mathcal{Q})=s(\wedge^2\mathcal{S}^*)\pi^*c(\wedge^2\mathcal{E}).$$
Equation (\ref{projection}) can thus be rewritten as
$$P_{I}(\mathcal{E})=c(\wedge^2\mathcal{E})\pi_*(s(\wedge^2\mathcal{S}^*)s_{\lambda(I)}(\mathcal{Q}))_{|deg=\Sigma I},$$
where the last symbols mean we only keep the component of degree $\sum I$. 

Now suppose that $\mathcal{E}=\mathcal{E}_0\otimes L$ for a line bundle $L$ and a trivial vector bundle $\mathcal{E}_0$. Then
$$c(\wedge^2\mathcal{E})=(1+2\delta)^{\binom{s}{2}}.$$
Moreover $G^r(\mathcal{E})$ is a trivial bundle over $X$, while $\mathcal{S}=\mathcal{S}_0\otimes L$ and $\mathcal{Q}=\mathcal{Q}_0\otimes L$
are obtained by pulling-back the tautological and quotient bundles $\mathcal{S}_0$, $\mathcal{Q}_0$ over a fixed Grassmannian 
$G^r(\mathbf{C}^n)$ (we omit the pull-backs for simplicity). By Definition \ref{def:sIJ} 
we have:
$$s_{\lambda(I)}(\mathcal{Q})=\sum_{J\le I}s_{I,J}s_{\lambda(J)}(\mathcal{Q}_0)\delta^{\Sigma I-\Sigma J},$$
where $\delta=c_1(L)$. Moreover, the Segre classes of $\wedge^2\mathcal{S}_0^*$ and  $\wedge^2\mathcal{S}^*$%
are related by the formula 
$$s(\wedge^2\mathcal{S}^*)=\sum_{\ell\ge 0} (1+2\delta)^{-\binom{r}{2}-\ell}s_\ell(\wedge^2\mathcal{S}_0^*).$$
Plugging these two formulas into the previous one, we get $P_{I}(\mathcal{E})$ as 
$$\sum_{J\le I}\sum_L(1+2\delta)^{\binom{n}{2}-\binom{r}{2}-|\lambda(L)|}\delta^{\Sigma I-\Sigma J}s_{I,J}\alpha_L
\pi_*(s_{\lambda(L)}(\mathcal{S}_0^*)s_{\lambda(J)}(\mathcal{Q}_0))_{|deg=\Sigma I}.$$
Now recall that the Schur classes $s_\alpha(\mathcal{S}_0^*)$ and $s_\beta(\mathcal{Q}_0)$, 
for partitions $\alpha\subset (r^s)$ and $\beta\subset (s^r)$, that are non zero, give  dual bases of Schubert cycles on the 
Grassmannian $G^r(\mathbf{C}^n)$. This can be expressed as 
$$\pi_*(s_{\lambda(L)}(\mathcal{S}_0^*)s_{\lambda(J)}(\mathcal{Q}_0))=\delta_{L,[n]/J},$$
where $\delta_{L,[n]/J}$ is the Kronecker delta.
Note that $L=[n]/J$ implies that $|\lambda(L)|+|\lambda(J)|=rs$. We thus get the formula 
$$P_{I}(\mathcal{E})=\Big(\sum_{J\le I}
(1+2\delta)^{\Sigma J}\delta^{\Sigma I-\Sigma J}s_{I,J}\alpha_{[n]/J}\Big)_{|deg=\Sigma I}.$$
But since the degree of the polynomial into brackets is exactly $\Sigma I$, we just need to keep its 
top degree component, that is 
$$P_{I}(\mathcal{E})=\sum_{J\le I}2^{\Sigma J}s_{I,J}\alpha_{[n]/J} \delta^{\Sigma I}.$$
We conclude by applying formally this formula to the bundle $E=\mathcal{O}(1/2)^{\oplus n}$ over the projective space. 
\end{proof}

\begin{thm}\label{nrsD} 
Let $m,n,r$ be positive integers. Then
\[\delta_D(m,n,n-r)=
\sum_{\sum I\le m} (-1)^{m-\sum I}\alpha_I d_I(n)\binom {m-1} {m-\sum I}\]
where the sum goes through all sets of nonnegative integers of cardinality $r$. 
\end{thm}

\begin{proof}
We replace $d_I(n)$ by the expression from Lemma \ref{lemmab_I}, change the order of summation and use Lemma \ref{lemmasijTypeD} in the last step:
\begin{align*}
&\sum_{\sum I\le m} (-1)^{m-\sum I} \alpha_I d_I(n)\binom {m-1} {m-\sum I}=\\ &
\hspace*{2cm}=\sum_{\sum I\le m} \sum_{J\le I} s_{I,J} \alpha_{[2n]\setminus J} \left(\frac{1}{2}\right)^{\sum I-\sum J} (-1)^{m-\sum I} \alpha_I \binom {m-1}{m-\sum I}
\\&\hspace*{2cm} =\sum_{\sum J\le m}(-1)^{m-\sum J} \alpha_{[2n]\setminus J}\sum_{\substack{I\ge J\\ \sum I\le m}}s_{I,J} \left(-\frac{1}{2}\right)^{\sum I-\sum J}  \alpha_I\binom {m-1}{m-\sum I}
\\&\hspace*{5cm} =\sum_{\sum J= m}\alpha_{[2n]\setminus J}\alpha_{J}=\delta_D(m,n,n-s).  
\end{align*}
\end{proof}

\section{Future directions and Conjectures}\label{sec:conj}

\subsection{Dual degrees of defective determinantal loci}

For symmetric matrices,
\Cref{thm:polyDelta} asserts that when $SD^{n-s,n}_{m}$ is not dual defective, that is, when Pataki's inequalities (\ref{eq:Pataki}) are satisfied,
its codegree depends polynomially on $n$ when $s$ and $m$ are fixed. By the fundamental duality relation 
(\Cref{rem:dualityDelta}),
this also means that the codegree of $SD^{s,n}_{\binom{n+1}{2}-m}$
depends polynomially on $n$ for $s$ and $m$ fixed.

What does happen in the 
defective cases? For example, by \cite{HarrisTu}, one has the formula 
$$\mathrm{codegree}(SD^{s,n}_{\binom{n+1}{2}-1})=\mathrm{degree}(SD^{n-s,n}_{\binom{n+1}{2}-1})=\prod_{j=1}^{s}\frac{\binom{n+j-1}{s-j+1}}{\binom{2j-1}{j-1}},$$
which is clearly polynomial in $n$. Note that in this case the dual defect is $\binom{s+1}{2}-1$, 
independently of $n$. Could it happen that
the following holds?
\medskip

\begin{conj}
For any fixed $m>0$,
	\begin{enumerate}
		\item the dual defect of $SD^{s,n}_{\binom{n+1}{2}-m}$ is equal to 
		$\max(0, \binom{s+1}{2}-m)$, independently of $n$, 
		\item the codegree of $SD^{s,n}_{\binom{n+1}{2}-m}$ depends polynomially on $n$. 
\end{enumerate}
\end{conj}

Similar statements should hold for general and skew-symmetric matrices. For general matrices, 
the degrees of the determinantal loci have been computed in \cite{HarrisTu}. In particular, 
this implies that  
$$\mathrm{codegree}(D^{s,n}_{n^2-1})=\mathrm{degree}(D^{n-s,n}_{n^2-1})=\prod_{j=1}^{s}\frac{\binom{n+j-1}{2j-1}}{\binom{s+j-1}{2j-1}}$$
is polynomial in $n$. In this case the dual defect is $s^2-1$, 
independently of $n$. Could it happen that
the following holds?
\medskip

\begin{conj}
For any fixed $m>0$,
	\begin{enumerate}
		\item the dual defect of $D^{s,n}_{n^2-m}$ is equal to 
		$\max(0, s^2-m)$, independently of $n$, 
		\item the codegree of $D^{s,n}_{n^2-m}$ depends polynomially on $n$. 
\end{enumerate}
\end{conj}

Finally, for skew-symmetric matrices, 
the degrees of the determinantal loci have also been computed in \cite{HarrisTu}. In particular, 
this implies that  
$$\mathrm{codegree}(AD^{2s,2n}_{\binom{2n}{2}-1})=\mathrm{degree}(AD^{2n-2s,2n}_{\binom{2n}{2}-1})=
\frac{1}{2^{2s-1}}\prod_{j=1}^{2s-1}\frac{\binom{2n+j-1}{2s-j}}{\binom{2j-1}{j-1}}$$
is polynomial in $n$. In this case the dual defect is $\binom{2s}{2}-1$, 
independently of $n$. Could it happen that
the following holds, and a similar statement in odd dimensions?
\medskip

\begin{conj}
For any fixed $m>0$,
	\begin{enumerate}
		\item the dual defect of $AD^{2s,2n}_{\binom{2n}{2}-m}$ is equal to 
		$\max(0, \binom{2n}{2}-m)$, independently of $n$, 
		\item the codegree of $AD^{2s,2n}_{\binom{2n}{2}-m}$ depends polynomially on $n$. 
\end{enumerate}
\end{conj}

 \subsection{Dual degrees of singular varieties}
 
 There is a general Pl\"ucker formula for the degree of the dual variety $X^*$ of a possibly singular
  projective variety $X\subset\PP^n$ \cite[Theorem 1.1]{ernstrom1997plucker}:
 \[\deg X^*=(-1)^n\left(\Chi(\Eu_X)-2\Chi(\Eu_{X_1})+\Chi(\Eu_{X_2})\right),\]
 where $X_1$ (resp.~$X_2$) is a general hyperplane section  (resp.~codimension two linear section) of $X$. 
 Here $\Eu(X)$ denotes the \emph{Euler obstruction} of $X$ \cite[Example 4.2.9 and Example 19.1.7]{fulton}. 
 Our Theorem \ref{nrs} confirming the NRS conjecture, and their analogues for the other types of 
 determinantal loci, would immediately follow if one could prove that Pragacz's formulas \cite{pragacz1988enumerative} for the topological
 Euler characteristics of determinantal loci, which a priori are only valid in the smooth case, can also be 
used in the 
 singular case, and compute the Euler characteristics of the Euler obstruction. This is an approach
 we plan to investigate in the near future, and could potentially apply also to the defective cases.

\subsection{Polynomiality of intersection products}
\Cref{thm:polyPHI} shows that for fixed $d$, the product
\[
\phi(n,d) = \int_{\CQ_n}L_{1}^{d-1} L_{n-1}^{\binom{n+1}{2}-d}
\]
is a polynomial in $n$. 
Can this result be generalized to other intersection products on $\CQ$? More precisely, we have the following question:
\begin{question}
	Let $d_1,\ldots,d_k \in \N$ with $\sum{d_i}=d-1$. Is 
	\[
	\int_{\CQ_n}L_{1}^{d_1}L_{2}^{d_2}\cdots L_{k}^{d_k} L_{n-1}^{\binom{n+1}{2}-d}
	\]
	a polynomial in $n>k+1$?
	
More generally: by which cohomology classes can one replace $L_{1}^{d_1}L_{2}^{d_2}\cdots L_{k}^{d_k}$ for the polynomiality property to hold?
\end{question}

\subsection{Representation theory}
The version of the NRS conjecture that we proved in type A (\Cref{thm:NRSA}) expresses
the codegree $\delta_A(m,n,n-s)$ as a linear combination of dimensions of Schur modules 
of $GL(n,\mathbb{C})$, with highest weights and multiplicities depending only on $m$ and 
$r$. In other words, it is obtained as the dimensional evaluation of a fixed character, 
depending only on $m$ and $s$. A natural question is: what is really this character? 
Is there a natural (combinatorial, or geometric) interpretation of the corresponding 
representations? 

The same question can be raised both in types B and D, where the dependence in $n$ of the 
codegrees $\delta(m,n,n-s)$ and $\delta_D(m,n,n-s)$ only appears through the number of 
one's one which a certain combination of Q-Schur or P-Schur functions are evaluated. We 
mentionned in the introduction that these evaluations of  Q-Schur or P-Schur functions 
count certain types of shifted tableaux. An alternative interpretation is that they 
give the dimensions of certain representations of the queer Lie super-algebra $\mathfrak{q}(n)$ 
(see for example \cite[Theorem 4.11]{Brundan}). 
So $\delta(m,n,n-s)$ and $\delta_D(m,n,n-s)$ can also be interpreted as dimensions of 
certain representations of $\mathfrak{q}(n)$, whose characters only depend on $m$ and $s$. 
What are these representations? Do they admit natural constructions or interpretations?

\subsection{Noncommutative matroids}
We believe that our results may also be viewed in the context of noncommutative matroids. Indeed, 
suppose we restrict to diagonal matrices, instead of symmetric. One may still consider the rational 
map given by inverting matrices, which gives the classical Cremona transformation. The famous 
resolution of that graph given by the permutohedral variety is the analog of the variety of complete 
quadrics. A representable matroid may be viewed as a subspace of the space of diagonal matrices. 
Many interesting invariants of this matroid may be read from the cohomology class of its strict 
transform in the permutohedral variety. In this analogy, the representable symmetric noncommutative 
matroid, whould be a subspace of symmetric matrices. Its crucial invariants should come from the 
cohomology class of the strict transform of that subspace to the variety of complete quadrics. 
For future work, it would be interesting to dare to define noncommutative matroids as special 
cohomology classes of the variety of complete quadrics.

\subsection{Algebraic statistics}
In terms of algebraic statistics, the number $\phi(n,d)$ is equal to the maximum likelihood degree of a general \emph{linear concentration model}. A related quantity is the maximum likelihood degree of a general \emph{linear covariance model}.
For $\fL \subseteq S^2(\CC^n)$, the ML-degree of the associated linear concentration model is the number of pairs $(\Sigma,K) \in S^2(\CC^n)^2$ satisfying
\[
\Sigma \cdot K = \Id_n \text{,             } K \in \fL \text{,            } \Sigma - S \in \fL^\perp,
\]
where $S$ is generic. In contrast, the ML-degree of the associated linear covariance model is the number of pairs $(\Sigma,K) \in S^2(\CC^n)^2$ for which
\[
\Sigma \cdot K = \Id_n \text{,             } \Sigma \in \fL \text{,            } KSK-K \in \fL^\perp.
\]
 In \cite{STZ}, it was conjectured that for generic $\fL$, these ML-degrees also are a polynomial in $n$. In future work, we plan to apply our geometric methods in order to prove this conjecture.

\bibliography{bibML}
\bibliographystyle{plain}
\end{document}